\documentclass{amsart}
\usepackage{amssymb}

\newtheorem{theorem}{Theorem}[section]
\newtheorem{lemma}[theorem]{Lemma}
\newtheorem{proposition}[theorem]{Proposition}

\theoremstyle{definition}
\newtheorem{definition}[theorem]{Definition}

\theoremstyle{remark}
\newtheorem{remark}[theorem]{Remark}

\numberwithin{equation}{section}

\newcommand{\JM}{Mierczy\'nski}
\newcommand{\RR}{\ensuremath{\mathbb{R}}}

\newcommand{\ZZ}{\ensuremath{\mathbb{Z}}}
\newcommand{\NN}{\ensuremath{\mathbb{N}}}
\newcommand{\PP}{\ensuremath{\mathbb{P}}}
\newcommand{\TT}{\ensuremath{\mathbb{T}}}

\newcommand{\OFP}{\ensuremath{(\Omega,\mathfrak{F},\PP)}}

\newcommand{\abs}[1]{\ensuremath{\lvert#1\rvert}}
\newcommand{\intpart}[1]{\ensuremath{\lfloor#1\rfloor}}
\newcommand{\norm}[1]{\ensuremath{\lVert#1\rVert}}

\newcommand{\mlsps}{measurable linear skew\nobreakdash-\hspace{0pt}product
semidynamical system}

\DeclareMathOperator{\cl}{cl}
\DeclareMathOperator{\codim}{codim}

\DeclareMathOperator{\Id}{Id}

\DeclareMathOperator{\Int}{Int}

\DeclareMathOperator{\lnplus}{ln^{+}}

\DeclareMathOperator{\osc}{osc}

\DeclareMathOperator{\spanned}{span}

\begin{document}
\title[Principal Lyapunov Exponents and Principal Floquet Spaces]
{Principal Lyapunov Exponents and Principal Floquet Spaces of
Positive Random Dynamical Systems. I. General Theory}

\author{Janusz Mierczy\'nski} \address{Institute of Mathematics and
Computer Science, Wroc{\l}aw University of Technology, Wybrze\.ze
Wyspia\'nskiego 27, PL-50-370 Wroc{\l}aw, Poland}
\urladdr{http://www.im.pwr.wroc.pl/~mierczyn/index.html}
\email{mierczyn@pwr.wroc.pl}
\thanks{The first-named author was supported from resources for
science in years 2009-2012 as research project (grant MENII N N201
394537, Poland)}

\author{Wenxian Shen} \address{Department of Mathematics and
Statistics, Auburn University, Auburn University, AL 36849, USA}
\email{wenxish@auburn.edu}
\thanks{The second-named author was partially supported by NSF grant
DMS-0907752}

\date{}


\begin{abstract}
This series of papers is concerned with principal Lyapunov exponents
and principal Floquet subspaces of positive random dynamical systems
in  ordered Banach spaces.  The current part of the series focuses on
the development of general theory. First, the notions of generalized
principal Floquet subspaces, generalized principal Lyapunov
exponents, and generalized exponential separations for general
positive random dynamical systems in  ordered Banach spaces are
introduced, which extend the classical notions of principal Floquet
subspaces, principal Lyapunov exponents, and exponential separations
for strongly positive deterministic systems in strongly ordered
Banach to general positive random dynamical systems in  ordered
Banach spaces. Under some quite general assumptions, it is then shown
that a positive random dynamical system in an ordered Banach space
admits a family of generalized principal Floquet subspaces, a
generalized principal Lyapunov exponent, and a generalized
exponential separation. We  will consider in the forthcoming part(s)
the applications of the general theory developed in this part to
positive random dynamical systems arising from a variety of random
mappings and differential equations, including random Leslie matrix
models, random cooperative systems of ordinary differential
equations,  and random parabolic equations.
\end{abstract}


\maketitle

\section{Introduction}
\label{section-introduction}

This is the first part of a series of papers. The series is devoted
to the study of principal Lyapunov exponents and principal Floquet
subspaces of positive random dynamical systems in  ordered Banach
spaces.  This first part focuses on the development of general theory
of principal Lyapunov exponents and principal Floquet subspaces of
general positive random dynamical systems in ordered Banach spaces.
The forthcoming part(s) of the series  will concern the applications
of the general theory developed in Part I to positive random
dynamical systems arising from a variety of random mappings and
differential equations, including random Leslie matrix models, random
cooperative systems of ordinary differential equations,  and random
parabolic equations.

Lyapunov exponents play an important role in the study of asymptotic
dynamics of linear and nonlinear random evolution systems. The study
of Lyapunov exponents traces back to Lyapunov \cite{Lya}. Oseledets
in \cite{Ose} obtained some important results on Lyapunov exponents
for finite dimensional systems, which is called now the Oseledets
multiplicative ergodic theorem. Since then, a huge amount of research
has been carried out toward alternative proofs of the Osedelets
multiplicative ergodic theorem and extensions of the Osedelets
multiplicative theorem for finite dimensional systems to certain
infinite dimensional ones (see \cite{Arn}, \cite{JoPaSe}, \cite{Kre},
\cite{Lian-Lu},  \cite{Man}, \cite{Mil}, \cite{Raghu}, \cite{Rue1},
\cite{Rue2}, \cite{ScFl}, and references therein).  In the recent
work \cite{Lian-Lu},  Lian and Lu  studied Lyapunov exponents of
general infinite dimensional random dynamical systems in a Banach
space and established a multiplicative ergodic theorem for such
systems.

The largest finite Lyapunov exponents (or top Lyapunov exponents) and
the associated invariant subspaces of both deterministic and random
dynamical systems play special roles in the applications to nonlinear
systems.  Classically, the top finite Lyapunov exponent of a positive
deterministic or random dynamical system in an ordered Banach space
is called the {\em principal Lyapunov exponent\/} if its associated
invariant subspace is one dimensional and is spanned by a positive
vector (in such case, the invariant subspace is called the {\em
principal Floquet subspace}).  Principal Lyapunov exponents and
principal Floquet subspaces are the analog of principal eigenvalues
and principal eigenfunctions of elliptic and time periodic parabolic
operators. Numerous works have also been carried out toward principal
Lyapunov exponents and principal Floquet subspaces for certain
positive deterministic as well as random dynamical systems in ordered
Banach spaces, in particular, for  deterministic and random dynamical
systems generated by nonautonomous and random parabolic equations
with bounded coefficients (see \cite{Hu}, \cite{HuPo}, \cite{HuPoSa},
\cite{HuShVi}, \cite{Mi1}, \cite{Mi2},   \cite{MiSh1}, \cite{MiSh2},
\cite{MiSh3}, \cite{Po1}, \cite{ShVi}, \cite{ShYi}, and references
therein).

Many strongly positive deterministic as well as random dynamical
systems in strongly ordered Banach space  are shown to have principal
Lyapunov exponents  (and  hence principal Floquet subspaces and
entire positive orbits). Moreover, the so called exponential
separations are admitted in such systems. For example, let $(Z,
(\sigma_t)_{t\in\RR})$ be a compact uniquely ergodic minimal flow and
$X$ be a strongly ordered Banach space with the positive cone $X^+$
(see \ref{subsection-ordered} for detail). Let $\Pi = (\Pi_t)_{t \ge
0}$, $\Pi_t \colon X \times Z \to X \times Z$ be a
skew\nobreakdash-\hspace{0pt}product semiflow over
$(Z,(\sigma_t)_{t\in\RR})$,
\begin{equation*}
\Pi_t(x,z) = (\Phi(t,z)x, \sigma_{t} z),
\end{equation*}
where $\Phi(t,z) \in \mathcal{L}(X,X)$.  If $\Pi$ is strongly
positive (i.e. $\Phi(t,z)x \in \Int{(X^+)}$ for any $t > 0$, $z \in
Z$, and $x \in X^+ \setminus \{0\}$) and completely continuous (i.e.,
$\{\, \Phi(t,z)B : z \in Z \,\}$ is a relatively compact subset of
$X$ for any $t > 0$ and any bounded subset $B$ of $X$), then there
are $\lambda_1 \in \RR$, $M, \gamma > 0$,  a subspace $E(z) \subset
X$ with $E(z) = \spanned{\{v(z)\}}$ for some $v(z) \in \Int{(X^+)}$,
$\norm{v(z)} = 1$, and a subspace $F(z) \subset X$  with $F(z)\cap
X^+=\{0\}$ such that $X = E(z) \oplus F(z)$ for any $z \in Z$, $E(z)$
and $F(z)$ are continuous in $z \in Z$, and
\begin{itemize}
\item[{\rm (i)}]
$\Phi(t,z)E(z) = E(\sigma_{t} z)$ for any $t > 0$ and $z \in Z$;
\item[{\rm (ii)}]
$\Phi(t,z)F(z) \subset F(\sigma_{t} z)$ for any $t > 0$ and $z
\in Z$;
\item[{\rm (iii)}]
$\displaystyle \lim_{t\to\infty}
\frac{\ln{\norm{\Phi(t,z)v(z)}}}{t} = \lambda_1$;
\item[{\rm (iv)}]
$\displaystyle \frac{\norm{\Phi(t,z) w}}{\norm{\Phi(t,z)v(z)}}
\le M e^{-{\gamma} t}$ for any $w \in F(z)$ with $\norm{w} = 1$,
$t > 0$, and $z \in Z$
\end{itemize}
(see \cite{MiSh1}, \cite{PoTe1}). Here $\lambda_1$ and
$\{E(z)\}_{z\in Z}$ are the principal Lyapunov exponent and principal
Floquet subspaces of $\Pi$, respectively, and the property (iv) is
referred to the {\em exponential separation} of $\Pi$.  Note that the
above results extend the classical Kre\u{\i}n--Rutman theorem for
strongly positive and compact operators in strongly order Banach
spaces to strongly positive and compact deterministic
skew\nobreakdash-\hspace{0pt}product semiflows in strongly ordered
Banach spaces.

For a general positive random dynamical system, there may be no
finite Lyapunov exponents (and hence no principal Lyapunov exponent
in classical sense); if the top Lyapunov exponent is finite, its
associated invariant subspace may not be one dimensional (and hence
there is no principal Lyapunov exponent in the classical sense
either).  It is not known whether a general positive random dynamical
system admits positive entire orbits and/or  invariant subspaces
spanned by positive vectors.

The objective of the current part of the series  is to investigate
the extent to which  the principal Lyapunov exponents and principal
Floquet subspaces theory for strongly positive and compact
deterministic dynamical systems may be generalized to general
positive random dynamical systems. The classical Kre\u{\i}n--Rutman
theorem for strongly positive and compact operators in strongly
ordered Banach spaces is extended to quite general positive random
dynamical systems in ordered Banach spaces.  In~particular, the
existence of entire positive orbits is shown without the assumption
of strong positivity (see Theorem
\ref{thm-largest-meets-nontrivially}); the existence of one
dimensional invariant measurable subspaces which are spanned by
positive vectors and whose associated Lyapunov exponent is the
largest  (such invariant subspaces and the associated Lyapunov
exponent are called {\em generalized principal Floquet subspaces\/}
and {\em generalized principal Lyapunov exponents\/}, respectively,
see Definition \ref{generalized-floquet-space}) is proved without the
assumption of the existence of finite Lyapunov exponents (see Theorem
\ref{theorem-w}), and the existence of a {\em generalized exponential
separation\/} (see Definition
\ref{generalized-exponential-separation}) is proved too without the
assumption of the existence of finite Lyapunov exponents (see Theorem
\ref{separation-thm}).

In the forthcoming part(s) of this series we will study the
applications of the general results established in this part to
random Leslie matrix models, random cooperative systems of ordinary
differential equations, and random parabolic equations.

The rest of the current part is organized as follows. In
Section~\ref{section-notation}, we introduce standing notions and
assumptions.  We introduce the concepts of generalized principal
Floquet subspaces, generalized principal Lyapunov exponents, and
generalized exponential separations and state the main results of
this part in Section~\ref{section-main-results}.  In
Section~\ref{section-preliminaries}, we present some preliminary
materials to be used in the proofs of the main results, including
some classical ergodic theorems and fundamental properties of Hilbert
projective metric. We prove the main results in the last section.

\section{Standing Notions and Assumptions}
\label{section-notation}
In this section, we introduce standing notions and assumptions.

If $f \colon A \to \RR$ we define
\begin{equation*}
f^{+}(a) :=
\begin{cases}
f(a) & \text{ if } f(a) \ge 0 \\
0 & \text{ if } f(a) < 0
\end{cases}
\qquad \text{and} \qquad f^{-}(a) :=
\begin{cases}
0 & \text{ if } f(a) \ge 0 \\
-f(a) & \text{ if } f(a) < 0
\end{cases}.
\end{equation*}
We have $f = f^{+} - f^{-}$ and $\abs{f} = f^{+} + f^{-}$.

\smallskip
For a metric space $Y$, $\mathfrak{B}(Y)$ stands for the
$\sigma$\nobreakdash-\hspace{0pt}algebra of all Borel subsets of $Y$.

\subsection{Metric Dynamical Systems}
\label{subsection-metric}

By a {\em probability space\/} we understand a triple $\OFP$, where
$\Omega$ is a set, $\mathfrak{F}$ is a
$\sigma$\nobreakdash-\hspace{0pt}algebra of subsets of $\Omega$, and
$\mathbb{P}$ is a probability measure defined for all $F \in
\mathfrak{F}$.

Let $\TT$ stand for either $\ZZ$ or $\RR$.

A {\em measurable dynamical system\/} on the probability space $\OFP$
is a $(\mathfrak{B}(\TT) \otimes \mathfrak{F},
\mathfrak{F})$\nobreakdash-\hspace{0pt}measurable mapping $\theta
\colon \TT \times \Omega \to \Omega$ such that
\begin{itemize}
\item
$\theta(0,\omega) = \omega$ for any $\omega \in \Omega$,
\item
$\theta(t_1 + t_2,\omega) = \theta(t_2, \theta(t_1,\omega))$ for
any $t_1, t_2 \in \TT$ and any $\omega \in \Omega$.
\end{itemize}
We write $\theta(t,\omega)$ as $\theta_{t}\omega$.  Also, we usually
denote measurable dynamical systems by $(\OFP, \allowbreak
(\theta_t)_{t\in\TT})$, or simply by $(\theta_t)_{t\in\TT}$.

A {\em metric dynamical system\/} is a measurable dynamical system
$(\OFP, \allowbreak (\theta_t)_{t\in\TT})$ such that for each $t \in
\TT$ the mapping $\theta_t \colon \Omega \to \Omega$ is
$\PP$\nobreakdash-\hspace{0pt}preserving (i.e.,
$\PP(\theta_t^{-1}(F)) = \PP(F)$ for any $F \in \mathfrak{F}$ and $t
\in \TT$).

When $\TT = \RR$ we call a (measurable, metric) dynamical system a
(measurable, metric) {\em flow\/}.  To emphasize the situation when
$\TT = \ZZ$, we speak of (measurable, metric) {\em
discrete\nobreakdash-\hspace{0pt}time\/} dynamical system.

When we use the symbol ``$\lim_{n \to \infty}$'' it is implied that
$n$ is considered for (perhaps sufficiently large) $n \in \NN$.
Similarly, when we use the symbol ``$\lim_{n \to -\infty}$'' it is
implied that $n$ is considered for (perhaps sufficiently large)
negative integers $n$.

For a measurable dynamical system $(\OFP,(\theta_t)_{t\in\TT})$,
$\Omega' \subset \Omega$ is {\em invariant\/} if $\theta_{t}(\Omega')
= \Omega'$ for all $t \in \TT$.  An $(\mathfrak{F},
\mathfrak{B}(\RR))$\nobreakdash-\hspace{0pt}measurable function
defined on an invariant $\Omega' \subset \Omega$ with $\PP(\Omega') =
1$ is called {\em invariant\/} if $f(\theta_{t}\omega) = f(\omega)$
for any $\omega \in \Omega'$ and any $t \in \TT$.

$(\OFP,(\theta_t)_{t\in\TT})$ is said to be {\em ergodic\/} if for
any invariant $F \in \mathfrak{F}$, either $\PP(F) = 1$ or $\PP(F) =
0$.

\subsection{Measurable Linear Skew-\hspace{0pt}Product Semidynamical
Systems}
\label{subsection-mlsps}

Let $X$ be a real Banach space, with norm $\norm{\cdot}$.  Let
$\mathcal{L}(X)$ stand for the Banach space of bounded linear
mappings from $X$ into $X$.  The standard norm in $\mathcal{L}(X)$
will be also denoted by $\norm{\cdot}$.

For a Banach space $X$, we will denote by $X^{*}$ its dual and by
$\langle \cdot, \cdot \rangle$ the standard duality pairing (that is,
for $u \in X$ and $u^{*} \in X^{*}$ the symbol $\langle u, u^{*}
\rangle$ denotes the value of the bounded linear functional $u^{*}$
at $u$).  Without further mention, we understand that the norm in
$X^{*}$ is given by $\norm{u^{*}} = \sup\{\,\abs{\langle u, u^{*}
\rangle}: \norm{u} \le 1\,\}$.

For $\TT = \RR$ we write $\TT^{+}$ for $[0, \infty)$.  For $\TT =
\ZZ$ we write $\TT^{+}$ for $\{0, 1, 2, 3, \dots\}$.

Let  $(\OFP, \allowbreak (\theta_t)_{t\in\TT})$ be a measurable
dynamical system. By a {\em measurable linear
skew\nobreakdash-\hspace{0pt}product semidynamical system\/} $\Phi =
((U_\omega(t))_{\omega \in \Omega, t \in \TT^{+}}, \allowbreak
(\theta_t)_{t\in\TT})$ on a Banach space $X$ covering
$(\theta_{t})_{t \in \TT}$ we understand a $(\mathfrak{B}(\TT^{+})
\otimes \mathfrak{F} \otimes \mathfrak{B}(X),
\mathfrak{B}(X))$\nobreakdash-\hspace{0pt}measurable mapping
\begin{equation*}
[\, \TT^{+} \times \Omega \times X \ni (t,\omega,u) \mapsto
U_{\omega}(t)u \in X \,]
\end{equation*}
satisfying the following:
\begin{itemize}
\item
\begin{equation}
\label{eq-identity}
U_{\omega}(0) = \mathrm{Id}_{X} \quad \forall
\omega \in \Omega,
\end{equation}
\begin{equation}
\label{eq-cocycle}
U_{\theta_{s}\omega}(t) \circ U_{\omega}(s) = U_{\omega}(t+s)
\qquad \forall \omega \in \Omega, \ t, s \in \TT^{+};
\end{equation}
\item
for each $\omega \in \Omega$ and $t \in \TT^{+}$, $[\, X \ni u
\mapsto U_{\omega}(t)u \in X \,] \in \mathcal{L}(X).$
\end{itemize}

When $\TT^{+} = [0,\infty)$ we call a \mlsps\ a (measurable linear
skew\nobreakdash-\hspace{0pt}product) {\em semiflow\/}.  To emphasize
the situation when $\TT^{+} = \{0, 1, 2, \dots \}$, we speak of
(measurable linear skew\nobreakdash-\hspace{0pt}product) {\em
discrete\nobreakdash-\hspace{0pt}time\/} semidynamical system.

If the Banach space $X$ is separable, by Pettis' theorem (see, e.g.,
\cite[Theorem~1.1.6]{Schw-Ye}), the measurability of the mapping $[\,
(t,\omega,u) \mapsto U_{\omega}(t)u \,]$ is equivalent to the fact
that for each $u^{*} \in X^{*}$ the mapping
\begin{equation*}
[\, \TT^{+} \times \Omega \times X \ni (t,\omega,u) \mapsto \langle
U_{\omega}(t)u, u^{*} \rangle \in \RR \,]
\end{equation*}
is $(\mathfrak{B}(\TT^{+}) \otimes \mathfrak{F} \otimes
\mathfrak{B}(X),
\mathfrak{B}(\RR))$\nobreakdash-\hspace{0pt}measurable.

\smallskip

For $\omega \in \Omega$, $t \in \TT^{+}$ and $u^{*} \in X^*$ we
define $U^{*}_{\omega}(t)u^{*}$ by
\begin{equation}
\label{dual-definition}
\langle u, U^{*}_{\omega}(t)u^{*} \rangle = \langle
U_{\theta_{-t}\omega}(t)u , u^{*} \rangle \qquad \text{for each } u
\in X
\end{equation}
(in other words, $U^{*}_{\omega}(t)$ is the mapping dual to
$U_{\theta_{-t}\omega}(t)$).  It is straightforward that
\begin{equation}
\label{eq-identity-dual}
U^{*}_{\omega}(0) = \mathrm{Id}_{X^{*}} \quad \text{for any } \omega
\in \Omega,
\end{equation}
and
\begin{equation}
\label{eq-cocycle-dual}
U^{*}_{\theta_{-s}\omega}(t) \circ U^{*}_{\omega}(s) =
U^{*}_{\omega}(t+s) \qquad \text{for any } \omega \in \Omega \text{
and any }t, s \in \TT^{+}.
\end{equation}

In case where the mapping
\begin{equation*}
[\, \TT^{+} \times \Omega \times X^{*} \ni (t,\omega,u^{*})
\mapsto U^{*}_{\omega}(t)u^{*} \in X^{*} \,]
\end{equation*}
is $(\mathfrak{B}(\TT^{+}) \otimes \mathfrak{F} \otimes
\mathfrak{B}(X^{*}),
\mathfrak{B}(X^{*}))$\nobreakdash-\hspace{0pt}measurable, we will
call the \mlsps\ $\Phi^{*} = ((U^{*}_\omega(t))_{\omega \in \Omega, t
\in \TT^{+}},(\theta_{-t})_{t\in\TT})$ on $X^{*}$ covering
$(\theta_{-t})_{t \in \TT}$ the {\em dual\/} of $\Phi$.

For instance, if we assume that $X$ is separable and reflexive, then
$X^{*}$ is separable, hence, by Pettis' theorem, the measurability of
the mapping $[\, (t,\omega,u^{*}) \mapsto U^{*}_{\omega}(t)u^{*} \,]$
is equivalent to the fact that for each $u \in X$ the mapping
\begin{equation*}
[\, \TT^{+} \times \Omega \times X^{*} \ni (t,\omega,u^{*}) \mapsto
\langle u, U^{*}_{\omega}(t)u^{*} \rangle = \langle
U_{\theta_{-t}\omega}(t)u , u^{*} \rangle\in \RR \,]
\end{equation*}
is $(\mathfrak{B}(\TT^{+}) \otimes \mathfrak{F} \otimes
\mathfrak{B}(X^{*}),
\mathfrak{B}(\RR))$\nobreakdash-\hspace{0pt}measurable, which in its
turn follows from the facts that the composition $[\, (t,\omega)
\mapsto (t, \theta_{-t}\omega) \mapsto U_{\theta_{-t}\omega}(t)u \,]$
is $(\mathfrak{B}(\TT^{+}) \otimes \mathfrak{F},
\mathfrak{B}(X))$\nobreakdash-\hspace{0pt}measurable and that the
$\langle \cdot, \cdot \rangle$ operation is continuous.

\smallskip
Since now till the end of the subsection we assume that
$((U_\omega(t))_{\omega \in \Omega, t \in \TT^{+}}, (\theta_{t})_{t
\in \TT^{+}})$ is a \mlsps.

Let $l$ be a positive integer. By a {\em family of
$l$\nobreakdash-\hspace{0pt}dimensional vector subspaces of $X$\/} we
understand a mapping $E$, defined on some $\Omega_0 \subset \Omega$
with $\PP(\Omega_0) = 1$, assigning to each $\omega \in \Omega_0$ an
$l$\nobreakdash-\hspace{0pt}dimensional vector subspace $E(\omega)$
of $X$.  Similarly, by a {\em family of
$l$\nobreakdash-\hspace{0pt}codimensional closed vector subspaces of
$X$\/} we understand a mapping $F$, defined on some $\Omega_0 \subset
\Omega$ with $\PP(\Omega_0) = 1$, assigning to each $\omega \in
\Omega_0$ an $l$\nobreakdash-\hspace{0pt}codimensional closed vector
subspace $F(\omega)$ of $X$.

We will usually denote families of vector subspaces by
$\{E(\omega)\}_{\omega \in \Omega_0}$, etc.

Regarding the measurability of families of
finite\nobreakdash-\hspace{0pt}dimensional vector subspaces, we will
use the following definition:  A family $\{E(\omega)\}_{\omega \in
\Omega_0}$ of $l$\nobreakdash-\hspace{0pt}dimensional vector
subspaces of $X$ is {\em measurable\/} if there are $(\mathfrak{F},
\mathfrak{B}(X))$\nobreakdash-\hspace{0pt}measurable functions $v_1,
\dots, v_l \colon \allowbreak \Omega_0 \to X$ such that
$(v_1(\omega), \dots, v_l(\omega))$ forms a basis of $E(\omega)$ for
each $\omega \in \Omega_0$ (see \cite[Lemma~5.6 and
Corollary~7.3]{Lian-Lu}).

Let $\{E(\omega)\}_{\omega \in \Omega_0}$ be a family of
$l$\nobreakdash-\hspace{0pt}dimensional vector subspaces of $X$, and
let $\{F(\omega)\}_{\omega \in \Omega_0}$ be a family of
$l$\nobreakdash-\hspace{0pt}codimensional closed vector subspaces of
$X$, such that $E(\omega) \oplus F(\omega) = X$ for all $\omega \in
\Omega_0$.  We define the {\em family of projections associated with
the decomposition\/} $E(\omega) \oplus F(\omega) = X$ as
$\{P(\omega)\}_{\omega \in \Omega_0}$, where $P(\omega)$ is the
linear projection of $X$ onto $F(\omega)$ along $E(\omega)$, for each
$\omega \in \Omega_0$.

The family of projections associated with the decomposition
$E(\omega) \oplus F(\omega) = X$ is called {\em strongly
measurable\/} if for each $u \in X$ the mapping $[\, \Omega_0 \ni
\omega \mapsto P(\omega)u \in X \,]$ is $(\mathfrak{F},
\mathfrak{B}(X))$\nobreakdash-\hspace{0pt}measurable.

We say that the decomposition $E(\omega) \oplus F(\omega) = X$, with
$\{E(\omega)\}_{\omega \in \Omega_0}$
finite\nobreakdash-\hspace{0pt}dimensional, is {\em invariant\/} if
$\Omega_0$ is invariant, $U_{\omega}(t)E(\omega) =
E(\theta_{t}\omega)$ and $U_{\omega}(t)F(\omega) \subset
F(\theta_{t}\omega)$, for each $t \in \TT^{+}$.

A strongly measurable family of projections associated with the
invariant decomposition $E(\omega) \oplus F(\omega) = X$ is referred
to as {\em tempered\/} if
\begin{equation*}
\lim\limits_{\substack{t \to \pm\infty \\ t \in \TT}}
\frac{\ln{\norm{P(\theta_{t}\omega)}}}{t} = 0 \qquad \PP\text{-a.s.
on }\Omega_0.
\end{equation*}

\subsection{Ordered Banach Spaces}
\label{subsection-ordered}

Let $X$ be a real Banach space, with norm $\norm{\cdot}$. By a {\em
cone\/} in   $X$ we understand a closed convex set $X^{+}$ such that
\begin{itemize}
\item[(C1)]
$\alpha \ge 0$ and $u \in X^{+}$ imply ${\alpha}u \in X^{+}$, and
\item[(C2)]
$X^{+} \cap (-X^{+}) = \{0\}$.
\end{itemize}

A pair $(X, X^{+})$, where $X$ is a Banach space and $X^{+}$ is a
cone in $X$, is referred to as an {\em ordered Banach space}.

If $(X, X^{+})$ is an ordered Banach space, for $u, v \in X$ we write
$u \le v$ if $v - u \in X^{+}$, and $u < v$ if $u \le v$ and $u \ne
v$.  The symbols $\ge$ and $>$ are used in analogous way.

We say that $u, v \in X^{+} \setminus \{0\}$ are {\em comparable\/},
written $u \sim v$, if there are positive numbers
$\underline{\alpha}, \overline{\alpha}$ such that
$\underline{\alpha}v \le u \le \overline{\alpha}v$.  The $\sim$
relation is clearly an equivalence relation.  For a nonzero $u \in
X^{+}$ we call the {\em component\/} of $u$, denoted by $C_{u}$, the
equivalence class of $u$, $C_{u} = \{\, v \in X^{+} \setminus \{0\}:
v \sim u \,\}$.

A cone $X^{+}$ in a Banach space $X$ is called
\begin{itemize}
\item
{\em solid \/}  if the interior $X^{++}$ of $X^{+}$ is nonempty,
\item
{\em reproducing\/} if $X^{+} - X^{+} = X$, and
\item
{\em total\/} if $X^{+} - X^{+}$ is dense in $X$.
\end{itemize}

An ordered Banach space $(X,X^+)$ is called {\it strongly ordered} if
$X^+$ is solid.

A cone $X^{+}$ in a Banach space $X$ is called {\em normal\/} if
there exists $K > 0$ such that for any $u, v \in X$ satisfying $0 \le
u \le v$ there holds $\norm{u} \le K \norm{v}$.

If $X^{+}$ is a normal cone we say that $(X, X^{+})$ is a {\em
normally ordered Banach space}.  In such a case, the Banach space $X$
can be renormed so that for any $u, v \in X$, $0 \le u \le v$ implies
$\norm{u} \le \norm{v}$ (see \cite[V.3.1, p.~216]{Schaef}).  Such a
norm is called {\em monotonic\/}.

From now on, when speaking of a normally ordered Banach space we
assume that the norm on $X$ is monotonic.

For an ordered Banach space $(X, X^{+})$ denote by $(X^{*})^{+}$ the
set of all $u^{*} \in X^{*}$ such that $\langle u, u^{*} \rangle \ge
0$ for all $u \in X^{+}$.  The closed subset $(X^{*})^{+}$ of $X^{*}$
is convex and satisfies (C1), however it need not satisfy (C2).
Nevertheless, if the cone $X^{+}$ is total then $(X^{*})^{+}$
satisfies (C2) (therefore is a cone).

It is a classical result that $X^{+}$ is normal if and only~if
$(X^{*})^{+}$ is reproducing, and that $X^{+}$ is reproducing if and
only~if $(X^{*})^{+}$ is normal, see e.g.~\cite[V.3.5]{Schaef}.

Sometimes an ordered Banach space $(X, X^{+})$ is a lattice:  any two
$u, v \in X$ have a least upper bound $u \vee v$ and a greatest lower
bound $u \wedge v$.  In such a case we write $u^{+} := u \wedge 0$,
$u^{-} := (-u) \vee 0$, and $\abs{u} := u^{+} + u^{-}$.  We have $u =
u^{+} - u^{-}$ for any $u \in X$.

An ordered Banach space $(X, X^{+})$ being a lattice is a {\em Banach
lattice\/} if there is a norm $\norm{\cdot}$ on $X$ (a {\em lattice
norm\/}) such that for any $u, v \in X$, if $\abs{u} \le \abs{v}$
then $\norm{u} \le \norm{v}$.  From now on, when speaking of a Banach
lattice we assume that the norm on $X$ is a lattice norm.

It is straightforward that in a Banach lattice the cone is normal and
reproducing.  Moreover, if $(X, X^{+})$ is a Banach lattice then
$(X^{*}, (X^{*})^{+})$ is a Banach lattice, too
(see~\cite[V.7.4]{Schaef}).

The reader is referred to forthcoming papers in the current series
for a variety of examples of ordered Banach spaces and ordered Banach
lattices.

\subsection{Assumptions}
\label{subsection-assumptions}

We list now assumptions we will make at various points in the sequel.

\medskip

\noindent\textbf{(A0)} (Ordered  Banach space) {\em
\begin{itemize}
\item[{\rm (i)}]
$(X, X^{+})$ is an ordered separable Banach space with $\dim{X}
\ge 2$.
\item[{\rm (ii)}]
$(X,X^+)$ is a normally ordered separable Banach space with
$\dim{X} \ge 2$.
\item[{\rm (iii)}]
$(X,X^+)$ is a separable Banach lattice with $\dim{X} \ge 2$.
\end{itemize}}

\medskip

\noindent\textbf{(A0)$^*$} (Ordered  Banach space) {\em
\begin{itemize}
\item[{\rm (i)}]
$(X^*, (X^*)^{+})$ is an ordered separable Banach space with
$\dim{X^*} \ge 2$.
\item[{\rm (ii)}]
$(X^*,(X^*)^+)$ is a normally ordered separable Banach space with
$\dim{X^*} \ge 2$.
\item[{\rm (iii)}]
$(X^*,(X^*)^+)$ is a separable Banach lattice with $\dim{X^*} \ge
2$.
\end{itemize}
}

\smallskip
\noindent \textbf{(A1)} (Integrability/injectivity/complete
continuity) {\em $\Phi = ((U_\omega(t))_{\omega \in \Omega, t \in
\TT^{+}}, \allowbreak (\theta_t)_{t\in\TT})$ is a \mlsps\ on a
separable Banach space $X$ covering an ergodic metric dynamical
system $(\theta_{t})_{t \in \TT}$ on $\OFP$, with the complete
measure $\PP$ in the case of $\TT = \RR$, satisfying the following:}
\begin{itemize}
\item[(i)] (Integrability) {\em
\begin{itemize}
\item
In the discrete\nobreakdash-\hspace{0pt}time case: The
function
\begin{equation*}
[\, \Omega \ni \omega \mapsto \lnplus{\norm{U_{\omega}(1)}}
\in [0,\infty) \,]
\end{equation*}
belongs to $L_1(\OFP)$.
\item
In the continuous\nobreakdash-\hspace{0pt}time case: The
functions
\begin{equation*}
[\, \Omega \ni \omega \mapsto \sup\limits_{0 \le s \le 1}
{\lnplus{\norm{U_{\omega}(s)}}} \in [0,\infty) \,]
\end{equation*}
and
\begin{equation*}
[\, \Omega \ni \omega \mapsto \sup\limits_{0 \le s \le 1}
{\lnplus{\norm{U_{\theta_{s}\omega}(1-s)}}} \in [0,\infty) \,]
\end{equation*}
belong to $L_1(\OFP)$.
\end{itemize}
}
\item[(ii)] (Injectivity)
{\em The linear operator $U_{\omega}(1)$ is injective almost
surely on $\Omega$.}
\item[(iii)] (Complete continuity)
{\em The linear operator $U_{\omega}(1)$ is completely continuous
almost surely on $\Omega$.}
\end{itemize}

In the sequel, by (A1)$^*$(i), (A1)$^*$(ii) and (A1)$^*$(iii) we will
understand the counterparts of (A1)(i), (A1)(ii) and (A1)(iii) for
the dual \mlsps\ $\Phi^{*}$.  More precisely, for~example
(A1)$^*$(ii) means the following: ``the mapping $[\, \TT^{+} \times
\Omega \times X^{*} \ni (t,\omega,u^{*}) \mapsto
U^{*}_{\omega}(t)u^{*} \in X^{*} \,]$ is $(\mathfrak{B}(\TT^{+})
\otimes \mathfrak{F} \otimes \mathfrak{B}(X^{*}),
\mathfrak{B}(X^{*}))$\nobreakdash-\hspace{0pt}measurable, and the
linear operator $U^{*}_{\omega}(1)$ is injective almost surely on
$\Omega$.''

\smallskip
Observe that, assuming the measurability in the definition of
$\Phi^{*}$ holds, if (A1)(i) is satisfied then (A1)$^*$(i) is
satisfied, too; similarly, if (A1)(iii) is satisfied then
(A1)(iii)$^*$ is satisfied.

%

\medskip
\noindent\textbf{(A2)} (Positivity) {\em $X$ satisfies
\textup{(A0)(i)} and $\Phi = ((U_\omega(t))_{\omega \in \Omega, t \in
\TT^{+}}, (\theta_t)_{t\in\TT})$ is a \mlsps\ on  $X$ covering an
ergodic metric dynamical system $(\theta_{t})_{t \in \TT}$ on $\OFP$,
satisfying the following:
\begin{equation*}
U_{\omega}(t)u_1 \le U_{\omega}(t)u_2
\end{equation*}
for any $\omega \in \Omega$, $t \in \TT^{+}$ and $u_1, u_2 \in X$
with $u_1 \le u_2$.}

\medskip
Similarly, we write

\medskip
\noindent\textbf{(A2)$^{*}$} (Positivity) {\em $X^*$ satisfies
\textup{(A0)$^*$(i)} and $\Phi^{*} = ((U^{*}_\omega(t))_{\omega \in
\Omega, t \in \TT^{+}}, (\theta_{-t})_{t\in\TT})$ is a \mlsps\ on
$X^{*}$ covering an ergodic metric dynamical system $(\theta_{-t})_{t
\in \TT}$ on $\OFP$, satisfying the following:
\begin{equation*}
U^{*}_{\omega}(t)u^{*}_1 \le U^{*}_{\omega}(t)u^{*}_2
\end{equation*}
for any $\omega \in \Omega$, $t \in \TT^{+}$ and $u^{*}_1, u^{*}_2
\in X^*$ with $u^{*}_1 \le u^{*}_2$.}

Observe that if (A0)$^*$(i) is satisfied and the measurability in the
definition of $\Phi^{*}$ holds then (A2) implies (A2)$^{*}$.

\medskip
\noindent\textbf{(A3)} (Focusing) {\em \textup{(A2)} is satisfied and
there are $\mathbf{e} \in X^{+}$  with $\norm{\mathbf{e}} = 1$  and
an $(\mathfrak{F},
\mathfrak{B}(\RR))$\nobreakdash-\hspace{0pt}measurable function
$\varkappa \colon \Omega \to [1,\infty)$ with
$\lnplus{\ln{\varkappa}} \in L_1(\OFP)$ such that for any $\omega \in
\Omega$ and any nonzero $u \in X^{+}$ there is $\beta(\omega,u) > 0$
with the property that
\begin{equation*}
\beta(\omega,u) \mathbf{e} \le U_{\omega}(1)u \le
\varkappa(\omega) \beta(\omega,u)\mathbf{e}.
\end{equation*}
}

\noindent\textbf{(A3)$^*$}  (Focusing)  {\em \textup{(A2)$^{*}$} is
satisfied and there are $\mathbf{e}^{*} \in (X^{*})^{+}$ with $
\norm{\mathbf{e}^{*}} = 1$  and an $(\mathfrak{F},
\mathfrak{B}(\RR))$\nobreakdash-\hspace{0pt}measurable function $
\varkappa^{*} \colon \Omega \to [1,\infty)$ with
$\lnplus{\ln{\varkappa^*}} \in L_1(\OFP)$ such that for any $\omega
\in \Omega$ and any nonzero $u^{*} \in (X^{*})^{+}$ there is
$\beta^{*}(\omega,u^{*}) > 0$ with the property that
\begin{equation*}
\beta^{*}(\omega,u^{*}) \mathbf{e}^{*} \le U^{*}_{\omega}(1)u^{*}
\le \varkappa^{*}(\omega) \beta^{*}(\omega,u^{*})\mathbf{e}^{*}.
\end{equation*}
}

\noindent\textbf{(A4)} (Strong focusing) {\em \textup{(A3)},
\textup{(A3)$^*$} are satisfied and  $\ln{\varkappa} \in L_1(\OFP)$,
$\ln{\varkappa^*} \in L_1(\OFP)$, and $\langle \mathbf{e},
\mathbf{e}^* \rangle > 0$.}

\medskip

\noindent\textbf{(A5)} (Strong positivity in one direction) {\em
There are $\mathbf{\overline{e}} \in X^+$ with
$\norm{\mathbf{\overline{e}}} = 1$ and an $(\mathfrak{F},
\mathfrak{B}(\RR))$\nobreakdash-\hspace{0pt}measurable function $\nu
\colon \Omega \to (0,\infty)$, with $\ln^{-}{\nu} \in L_1(\OFP)$,
such that
\begin{equation*}
U_{\omega}(1) \mathbf{\overline{e}} \ge \nu(\omega)
\mathbf{\overline{e}}\quad \forall\,\, \omega \in \Omega.
\end{equation*}
}

\noindent\textbf{(A5)$^*$} (Strong positivity in one direction) {\em
There are $\mathbf{\bar e^*}\in (X^*)^+$ with
$\norm{\mathbf{\overline{e}^*}} = 1$ and an $(\mathfrak{F},
\mathfrak{B}(\RR))$\nobreakdash-\hspace{0pt}measurable function
$\nu^* \colon \Omega \to (0,\infty)$, with $\ln^{-}{\nu^*} \in
L_1(\OFP)$, such that
\begin{equation*}
U_{\omega}^*(1) \mathbf{\overline{e}^*} \ge \nu^*(\omega)
\mathbf{\overline{e}^*} \quad \forall\,\, \omega \in \Omega.
\end{equation*}
}

\begin{remark}
We can replace time $1$ with some nonzero $T$ belonging to $\TT^{+}$
in (A1), (A3), (A4), (A5), and  (A1)$^*$, (A3)$^*$, (A5)$^*$.
\end{remark}


\begin{remark}
The focusing property in (A3) (resp. (A3)*) is an extension of the
so-called $u_0$\nobreakdash-\hspace{0pt}positivity for a
deterministic linear operator (see \cite{KeTr}, \cite{Kra}) to a
measurable linear skew-product semidynamical system.
\end{remark}

\section{Definitions and Main Results}
\label{section-main-results}
In this section, we state the definitions and  main results of the
paper. We first  state the definitions in
\ref{subsection-definition}, then recall an Oseledets-type Theorem
proved in \cite{Lian-Lu} in \ref{subsection-Oseledets}, and finally
state the main results in \ref{subsection-main-results}. Throughout
this section, we assume that $((U_\omega(t))_{\omega \in \Omega, t
\in \TT^{+}}, (\theta_{t})_{t \in \TT^{+}})$ is a \mlsps\ on a Banach
space $X$ covering $(\theta_{t})_{t \in \TT}$.

\subsection{Definitions}
\label{subsection-definition}

In this subsection, we introduce the concepts of entire solutions and
extend the notions of principal eigenfunction  and exponential
separation of strongly positive and compact operators. Throughout
this subsection, we assume (A0)(i) and (A2).

\begin{definition}[Entire  orbit]
For $\omega \in \Omega$, by an {\em entire orbit\/} of $U_{\omega}$
we understand a mapping $v_{\omega} \colon \TT \to X$ such that
$v_{\omega}(s + t) = U_{\theta_{s}\omega}(t) v_{\omega}(s)$ for any
$s \in \TT$ and $t \in \TT^{+}$.  The function constantly equal to
zero is referred to the {\em trivial entire orbit\/}.
\end{definition}

Entire orbits of $\Phi^*$ are defined in a  similar way.

\begin{definition}[Generalized principal Floquet subspaces and
principal Lyapunov exponent]
\label{generalized-floquet-space}
A family of one\nobreakdash-\hspace{0pt}dimensional subspaces
$\{\tilde{E}(\omega)\}_{\omega \in \tilde{\Omega}}$ of $X$ is called
a family of {\em generalized principal Floquet subspaces} of $\Phi =
((U_\omega(t))_{\omega \in \Omega, t \in \TT^{+}},
(\theta_t)_{t\in\TT})$ if $\tilde\Omega\subset \Omega$ is invariant,
$\PP(\tilde{\Omega}) = 1$, and
\begin{itemize}
\item[(i)]
$\tilde{E}(\omega) = \spanned{\{w(\omega)\}}$ with $w \colon
\tilde{\Omega} \to X^+ \setminus \{0\}$ being $(\mathfrak{F},
\mathfrak{B}(X))$\nobreakdash-\hspace{0pt}measurable,
\item[(ii)]
$U_{\omega}(t) \tilde{E}(\omega) = \tilde{E}(\theta_{t}\omega)$,
for any $\omega \in \tilde{\Omega}$ and any $t \in \TT^{+}$,
\item[(iii)]
there is $\tilde{\lambda} \in [-\infty, \infty)$ such that
\begin{equation*}
\tilde{\lambda} = \lim_{\substack{t\to\infty \\ t \in \TT^{+}}}
\frac{1}{t} \ln{\norm{U_\omega(t)w(\omega)}}
\end{equation*}
for any $\omega \in \tilde{\Omega}$, and
\item[(iv)]
\begin{equation*}
\limsup_{\substack{t\to\infty \\ t \in \TT^{+}}} \frac{1}{t}
\ln{\norm{U_\omega(t)u}} \le \tilde{\lambda}
\end{equation*}
for any $\omega \in \tilde{\Omega}$ and any $u \in X \setminus
\{0\}$.
\end{itemize}
$\tilde{\lambda}$ is called the {\em generalized principal Lyapunov
exponent} of $\Phi$ associated to the generalized principal Floquet
subspaces $\{\tilde E(\omega)\}_{\omega\in\tilde\Omega}$.
\end{definition}

Note that the notions of generalized principal Floquet subspaces and
principal Lyapunov exponent are the extensions of principal
eigenspaces and principal eigenvalues of strongly positive and
compact operators. If  $\Phi = ((U_\omega(t))_{\omega \in \Omega, t
\in \TT^{+}}, (\theta_t)_{t\in\TT})$  admits  a family of generalized
Floquet subspaces $\{\tilde{E}(\omega)\}_{\omega \in
\tilde{\Omega}}$, then $[\,\RR \ni t \mapsto U_\omega(t)w(\omega)\,]$
is a nontrivial entire positive orbit, where $U_\omega(t)w(\omega)$
is, for $t < 0$, understood as
$(U_{\theta_{t}\omega}(-t)|_{\tilde{E}_1(\theta_{t}\omega)})^{-1}
w(\omega)$.

\begin{definition}[Generalized exponential separation]
\label{generalized-exponential-separation}
$\Phi = ((U_\omega(t))_{\omega \in \Omega, t \in \TT^{+}},
\allowbreak (\theta_t)_{t\in\TT})$ is said to admit a {\em
generalized exponential separation\/} if there are a family of
generalized principal Floquet subspaces
$\{\tilde{E}(\omega)\}_{\omega \in \tilde{\Omega}}$ and a family of
one\nobreakdash-\hspace{0pt}codimensional subspaces
$\{\tilde{F}(\omega)\}_{\omega \in \tilde{\Omega}}$  of $X$
satisfying the following
\begin{itemize}
\item[(i)]
$\tilde F(\omega) \cap X^{+} = \{0\}$ for any $\omega \in
\tilde{\Omega}$,
\item[(ii)]
$X = \tilde{E}(\omega) \oplus \tilde{F}(\omega)$ for any
$\omega\in\tilde\Omega$, where the decomposition is invariant,
and the family of projections associated with this decomposition
is strongly measurable and tempered,
\item[(iii)]
there exists $\tilde{\sigma} \in (0, \infty]$ such that
\begin{equation*}
\lim_{\substack{t\to\infty \\ t \in \TT^{+}}} \frac{1}{t}
\ln{\frac{\norm{U_{\omega}(t)|_{\tilde{F}(\omega)}}}
{\norm{U_{\omega}(t) w(\omega)}}} = -\tilde{\sigma}
\end{equation*}
for each $\omega \in \tilde{\Omega}$.
\end{itemize}
We say that $\{\tilde{E}(\cdot), \tilde{F}(\cdot), \tilde{\sigma} \}$
{\em generates a generalized exponential separation}.
\end{definition}

We remark that in general the generalized principal Lyapunov exponent
$\tilde\lambda$ associated to the generalized principal Floquet
subspaces $\{\tilde E(\omega)\}_{\omega\in \tilde{\Omega}}$ may be
$-\infty$. The limit in Definition
\ref{generalized-exponential-separation} may not be uniform in
$\omega\in\tilde\Omega$.  The generalized exponential separation is
the extension of the classical exponential separation.

\subsection{Oseledets-type Theorem}
\label{subsection-Oseledets}

In this subsection, we recall an Oseledets-type  theorem  proved in
\cite{Lian-Lu}.
\begin{theorem}
\label{Oseledets-thm}
Let $X$ be a separable Banach space.  Let $\Phi$ be a \mlsps\
satisfying \textup{(A1)(i)--(iii)}. Then there exists an invariant
$\Omega_0 \subset \Omega$, $\PP(\Omega_0) = 1$, with the property
that one of the following \textup{(}mutually exclusive\textup{)}
cases, \textup{(1)}, \textup{(2)} or \textup{(3)}, holds:
\begin{itemize}
\item[{\rm (1)}] For each $\omega\in\Omega_0$,
\begin{equation*}
\lim\limits_{\substack{t\to\infty \\ t \in \TT^{+}}} \frac{1}{t}
\ln{\norm{U_{\omega}(t)}} = -\infty.
\end{equation*}
\item[{\rm (2)}]
There are: $k$ \textup{(}$k \ge 1$\textup{)} real numbers
$\lambda_1 > \dots > \lambda_k$, $k$ measurable families
$\{E_1(\omega)\}_{\omega\in\Omega_0}$, \dots,
$\{E_k(\omega)\}_{\omega \in \Omega_0}$, of vector subspaces of
finite dimensions, and a family
$\{F_{\infty}(\omega)\}_{\omega\in \Omega_0}$ of closed vector
subspaces of finite codimension such that
\begin{itemize}
\item[$\bullet$]
$U_{\omega}(t) E_{i}(\omega) = E_{i}(\theta_t\omega )$
\textup{(}$i = 1,2,\dots,k$\textup{)} and $U_{\omega}(t)
F_{\infty}(\omega) \subset F_{\infty}(\theta_t\omega)$, for
any $\omega\in\Omega_0$ and $t \in \TT^{+}$,
\item[$\bullet$]
$E_1(\omega) \oplus \dots \oplus E_k(\omega) \oplus
F_{\infty}(\omega) = X$ for any $\omega \in \Omega_0$;
moreover, the family of projections associated with the
decomposition $\Bigl(\bigoplus\limits_{j=1}^{i}
E_{j}(\omega)\Bigr) \oplus \Bigl(\bigoplus\limits_{j=i+1}^{k}
E_{j}(\omega) \oplus F_{\infty}(\omega)\Bigr) = X$
\textup{(}$i = 1,2,\dots,k$\textup{)} is strongly measurable
and tempered,
\item[$\bullet$]
\begin{equation*}
\lim\limits_{\substack{t\to\pm\infty \\ t \in \TT}}
\frac{1}{t} \ln{\norm{U_{\omega}(t)|_{E_i(\omega)}}} =
\lim\limits_{\substack{t\to\pm\infty \\ t \in \TT}}
\frac{1}{t} \ln{\norm{U_{\omega}(t)u}} = \lambda_i
\end{equation*}
for any $\omega \in \Omega_0$ and any nonzero $u \in
E_i(\omega)$ \textup{(}$i = 1,\dots, k$\textup{)},
\item[$\bullet$]
\begin{equation*}
\lim\limits_{\substack{t\to\infty \\ t \in \TT^{+}}}
\frac{1}{t} \ln{\norm{U_\omega(t)u}} = {\lambda}_i
\end{equation*}
for any $\omega \in \Omega_0$ and any $u \in (E_i(\omega)
\oplus E_{i+1}(\omega) \oplus \dots \oplus E_k(\omega) \oplus
F_{\infty}(\omega)) \setminus (E_{i+1}(\omega) \oplus \dots
\oplus E_k(\omega) \oplus F_{\infty}(\omega))$ \textup{(}$i =
1,\dots, k-1$\textup{)},
\item[$\bullet$]
\begin{equation*}
\lim\limits_{\substack{t\to\infty \\ t \in \TT^{+}}}
\frac{1}{t} \ln{\norm{U_\omega(t)u}} = {\lambda}_k
\end{equation*}
for any $\omega \in \Omega_0$ and any $u \in (E_k(\omega)
\oplus F_{\infty}(\omega)) \setminus F_{\infty}(\omega)$, and
\item[$\bullet$]
\begin{equation*}
\lim\limits_{\substack{t\to\infty \\ t \in \TT^{+}}}
\frac{1}{t} \ln{\norm{U_{\omega}(t)|_{F_{\infty}(\omega)}}} =
-\infty
\end{equation*}
for any $\omega \in \Omega_0$.
\end{itemize}
\item[{\rm (3)}]
There are a sequence of real numbers $\lambda_1 > \dots >
\lambda_i > \lambda_{i+1} > \dots {}$ having limit $-\infty$,
countably many measurable families $\{E_1(\omega)\}_{\omega \in
\Omega_0}$, $\{E_2(\omega)\}_{\omega \in \Omega_0}$, \dots, of
vector subspaces of finite dimensions, and countably many
families $\{F_1(\omega)\}_{\omega \in \Omega_0}$,
$\{F_2(\omega)\}_{\omega \in \Omega_0}$, \dots, of closed vector
subspaces of finite codimensions such that
\begin{itemize}
\item[$\bullet$]
$U_{\omega}(t) E_{i}(\omega) = E_{i}(\theta_t\omega)$ and
$U_{\omega}(t) F_{i}(\omega) \subset F_{i}(\theta_t\omega)$
\textup{(}$i = 1,2,\dots$\textup{)}, for any $\omega \in
\Omega_0$ and $t \in \TT^{+}$,
\item[$\bullet$]
$E_1(\omega) \oplus \dots \oplus E_i(\omega) \oplus
F_i(\omega) = X$ and $F_i(\omega) = E_{i+1}(\omega) \oplus
F_{i+1}(\omega)$ for any $\omega \in \Omega_0$ \textup{(}$i =
1, 2, \dots$\textup{)}; moreover, the family of projections
associated with the decomposition
$\Bigl(\bigoplus\limits_{j=1}^{i} E_{j}(\omega)\Bigr) \oplus
F_{i}(\omega) = X$ \textup{(}$i = 1,2,\dots$\textup{)} is
strongly measurable and tempered,
\item[$\bullet$]
\begin{equation*}
\lim\limits_{\substack{t\to\pm\infty \\ t \in \TT}}
\frac{1}{t} \ln{\norm{U_{\omega}(t)|_{E_i(\omega)}}} =
\lim\limits_{\substack{t\to\pm\infty \\ t \in \TT}}
\frac{1}{t} \ln{\norm{U_{\omega}(t)u}} = \lambda_i
\end{equation*}
for any $\omega \in \Omega_0$ and any nonzero $u \in
E_j(\omega)$ \textup{(}$i = 1, 2, \dots$\textup{)},
\item[$\bullet$]
\begin{equation*}
\lim\limits_{\substack{t\to\infty \\ t \in \TT^{+}}}
\frac{1}{t} \ln{\norm{U_\omega(t)u}} = {\lambda}_i
\end{equation*}
for any $\omega \in \Omega_0$ and any $u \in (E_i(\omega)
\oplus F_i(\omega)) \setminus F_{i}(\omega)$ \textup{(}$i =
1,2,\dots$\textup{)}, and
\item[$\bullet$]
\begin{equation*}
\lim\limits_{\substack{t\to\infty \\ t \in \TT^{+}}}
\frac{1}{t} \ln{\norm{U_{\omega}(t)|_{F_i(\omega)}}} =
\lambda_{i+1}
\end{equation*}
for any $\omega \in \Omega_0$ \textup{(}$i = 1, 2,
\dots$\textup{)}.
\end{itemize}
\end{itemize}
\end{theorem}
In the above, for $t = -s$ for some $s \in \TT^{+}$ and $u \in
E_{i}(\omega)$ the symbol $U_{\omega}(t)u$ stands for $v \in
E_{i}(\theta_{t}\omega)$ such that $U_{\theta_{t}\omega}(s)v = u$. In
view of the fact that $U_{\theta_{t}\omega}(s)
E_{i}(\theta_{t}\omega) = E_{i}(\omega)$ and the injectivity
(A1)(ii), such a $v$ is well defined.

In case (2), we write $F_{i}(\omega)$ for $E_{i+1}(\omega) \oplus
\dots \oplus E_{k}(\omega) \oplus F_{\infty}(\omega)$, $i = 1, 2,
\dots, k$.

In literature, $\lambda_i$'s in the cases (2) and (3) are called {\em
Lyapunov exponents\/} and $E_i(\omega)$'s are called the {\em
Oseledets spaces\/} associated to $\lambda_i$'s.

\subsection{Main results}
\label{subsection-main-results}

We state the main results of the paper in this subsection. The first
theorem is on the existence of entire positive orbits.

\begin{theorem}[Entire positive orbits]
\label{thm-largest-meets-nontrivially}
Assume $\Phi$ is a continuous \mlsps\ satisfying \textup{(A0)(i)},
\textup{(A1)(i)--(iii)} and \textup{(A2)}. If Theorem
\ref{Oseledets-thm}(2) or (3) occurs and $X^+$ is total then the set
$\Omega_1$ of those $\omega \in \Omega_0$ such that $E_1(\omega) \cap
X^{+} \varsupsetneq \{0\}$ has $\PP$\nobreakdash-\hspace{0pt}measure
one, and for each $\omega \in \Omega_1$ there exists an entire
positive  orbit $v_{\omega} \colon \TT \to X^+$ of $U_{\omega}$ such
that
\begin{equation*}
v_{\omega}(t) \in (E_1(\theta_{t}\omega) \cap X^{+}) \setminus
\{0\} \quad \forall t \in \TT.
\end{equation*}
\end{theorem}

The above theorem shows the existence of an entire positive orbit of
$U_\omega$ for a.e. $\omega\in\Omega$ without the assumption that
$U_\omega$ is strongly positive, which extends the principal
eigenfunction theory for strongly positive and compact operators.
Note that in general $E_1(\omega)\not ={\rm span}\{v_\omega(0)\}$ in
the case that Theorem \ref{Oseledets-thm}(2) or (3) occurs.

Next theorem shows the existence of generalized Floquet subspaces and
principal Lyapunov exponent and the  uniqueness of entire positive
orbits.

\begin{theorem}[Generalized principal Floquet subspace and Lyapunov exponent]
\label{theorem-w}
Assume \textup{(A0)(ii)}, \textup{(A1)(i)}, \textup{(A2)} and
\textup{(A3)}. Then there exist an invariant set $\tilde \Omega_1
\subset \Omega$, $\PP(\tilde{\Omega}_1) = 1$, and an $(\mathfrak{F},
\mathfrak{B}(X))$\nobreakdash-\hspace{0pt}measurable function $w
\colon \tilde{\Omega}_1 \to X$, $w(\omega) \in C_{\mathbf{e}}$ and
$\norm{w(\omega)} = 1$ for all $\omega \in \tilde{\Omega}_1$, having
the following properties:
\begin{itemize}
\item[{\rm (1)}]
\begin{equation*}
w(\theta_{t}\omega) =
\frac{U_{\omega}(t)w(\omega)}{\norm{U_{\omega}(t)w(\omega)}}
\end{equation*}
for any $\omega \in \tilde{\Omega}_1$ and $t \in \TT^{+}$.

\item[{\rm (2)}]
Let for some $\omega \in \tilde{\Omega}_1$ a function $v_{\omega}
\colon \TT \to X^{+} \setminus \{0\}$ be an entire orbit of
$U_{\omega}$.  Then $v_{\omega}(t) = \norm{v_{\omega}(0)}
w_{\omega}(t)$ for all $t \in \TT$, where
\begin{equation*}
w_{\omega}(t) := \begin{cases}
(U_{\theta_{t}\omega}(-t)|_{\tilde{E}_1(\theta_{t}\omega)})^{-1}
w(\omega) & \qquad \text{for } t \in \TT, \ t < 0 \\
U_{\omega}(t) w(\omega) & \qquad \text{for } t \in \TT^{+},
\end{cases}
\end{equation*}
with $\tilde E_1(\omega) = \spanned\{w(\omega)\}$.
\item[{\rm (3)}]
There exists $\tilde{\lambda}_1 \in [-\infty,\infty)$ such that
\begin{equation*}
\tilde{\lambda}_1 = \lim_{\substack{t\to \pm\infty \\ t \in
\TT}} \frac{1}{t} \ln{\rho_{t}(\omega)} =
\int\limits_{\Omega} \ln{\rho_{1}}\, d\PP
\end{equation*}
for each $\omega \in \tilde{\Omega}_1$, where
\begin{equation*}
\rho_t(\omega) := \begin{cases} \norm{U_\omega(t)w(\omega)} & \quad
\text{for } t \ge 0, \\
1/\norm{U_{\theta_{t}\omega}(-t)w(\theta_{t}\omega)} & \quad
\text{for } t < 0.
\end{cases}
\end{equation*}
\item[{\rm (4)}]
Assume, moreover, that \textup{(A1)(ii)--(iii)} hold and that
$X^+$ is total.  If Theorem \ref{Oseledets-thm}(1) occurs, then
$\tilde{\lambda}_1 = -\infty$. If Theorem \ref{Oseledets-thm}(2)
or (3) occurs, then $\tilde{\lambda}_1 = \lambda_1$.  Hence for
any $u \in X\setminus\{0\}$,
\begin{equation*}
\limsup_{\substack{t \to \infty \\ t \in \TT^{+}}} \frac{1}{t}
\ln{\norm{U_\omega(t)u}} \le \tilde{\lambda}_1,
\end{equation*}
and then $\{\tilde E_1(\omega)\}_{\omega \in \tilde{\Omega}_1}$
is a family of generalized Floquet subspaces.
\item[{\rm (5)}]
Assume, moreover, that \textup{(A0)(iii)} holds.  Then for any $u
\in X \setminus \{0\}$,
\begin{equation*}
\limsup_{\substack{t \to \infty \\ t \in \TT^{+}}} \frac{1}{t}
\ln{\norm{U_\omega(t)u}} \le \tilde{\lambda}_1,
\end{equation*}
and then $\{\tilde{E}_1(\omega)\}_{\omega \in \tilde{\Omega}_1}$
is a family of generalized Floquet subspaces.
\end{itemize}
\end{theorem}

Observe that $U_{\omega}(t) \tilde{E}_1(\omega) =
\tilde{E}_1(\theta_{t}\omega)$, for any $\omega \in \tilde{\Omega}_1$
and any $t \in \TT^{+}$.  Since $w \colon \tilde{\Omega}_1 \to X$ is
$(\mathfrak{F},
\mathfrak{B}(\RR))$\nobreakdash-\hspace{0pt}measurable,
$\{\tilde{E}_1(\omega)\}_{\omega \in \tilde{\Omega}_1}$ is a
measurable family of one\nobreakdash-\hspace{0pt}dimensional
subspaces of $X$. For $\omega \in \tilde{\Omega}_1$, the function
$w_{\omega} \colon \TT \to X^{+}$ is a nontrivial entire orbit of
$U_{\omega}$.  By Theorem \ref{theorem-w}(2), a nontrivial entire
orbit of $U_\omega$ is unique up~to multiplication by positive
scalar, which extends the fundamental property  on the existence and
uniqueness of positive eigenvectors of compact
$u_0$\nobreakdash-\hspace{0pt}positive linear operators (see
\cite{KeTr} and \cite{Kra}).  Note that Theorem
\ref{Oseledets-thm}(1) may occur under the assumptions of Theorem
\ref{theorem-w}.

The theorem below is a counterpart of Theorem~\ref{theorem-w}  for
the dual system.
\begin{theorem}[Generalized principal Floquet subspace and Lyapunov exponent]
\label{theorem-w-star}
Assume  \textup{(A0)$^*$(ii)}, \textup{(A1)$^{*}$(i)},
\textup{(A2)$^*$} and \textup{(A3)$^*$}. Then there exist an
invariant set $\tilde{\Omega}^{*}_1 \subset \Omega$,
$\PP(\tilde{\Omega}^{*}_1) = 1$, and an $(\mathfrak{F},
\mathfrak{B}(X^{*}))$\nobreakdash-\hspace{0pt}measurable function
$w^{*} \colon \tilde{\Omega}^{*}_1 \to X^{*}$, $w^{*}(\omega) \in
C_{\mathbf{e}^{*}}$ and $\norm{w^{*}(\omega)} = 1$ for all $\omega
\in \tilde{\Omega}^{*}_1$, having the following properties:
\begin{itemize}
\item[{\rm (1)}]
\begin{equation*}
w^{*}(\theta_{-t}\omega) =
\frac{U^{*}_{\omega}(t)w^{*}(\omega)}{\norm{U^{*}_{\omega}(t)w^{*}(\omega)}}
\end{equation*}
for any $\omega \in \tilde{\Omega}^{*}_1$ and $t \in \TT^{+}$.
\item[{\rm (2)}]
Let for some $\omega \in \tilde{\Omega}_1^*$ a function
$v_{\omega}^* \colon \TT \to (X^*)^{+} \setminus \{0\}$ be an
entire orbit of $U_{\omega}^*$.  Then $v_{\omega}^*(t) =
\norm{v_{\omega}^*(0)} w_{\omega}^*(t)$ for all $t \in \TT$,
where
\begin{equation*}
w_{\omega}^*(t) := \begin{cases}
(U_{\theta_{-t}\omega}^*(-t)|_{\tilde{E}_1^*(\theta_{-t}\omega)})^{-1}
w^*(\omega) & \qquad \text{for } t \in \TT, \ t < 0 \\
U_{\omega}^*(t) w^*(\omega) & \qquad \text{for } t \in \TT^{+},
\end{cases}
\end{equation*}
where $\tilde E_1^*(\omega)={\rm span}\{w^*(\omega)\}$.
\item[{\rm (3)}]
There exists $\tilde{\lambda}_1^* \in [-\infty,\infty)$ such that
\begin{equation*}
\tilde{\lambda}_1 ^* = \lim_{\substack{t\to \pm\infty \\ t \in
\TT}} \frac{1}{t} \ln{\rho_{t}^*(\omega)} = \int\limits_{\Omega}
\ln{\rho_{1}^*}\, d\PP
\end{equation*}
for each $\omega \in \tilde{\Omega}_1$, where
\begin{equation*}
\rho_t^*(\omega) := \begin{cases} \norm{U_\omega^*(t)w^*(\omega)} & \quad
\text{for } t \ge 0, \\
1/\norm{U_{\theta_{-t}\omega}^*(-t)w^*(\theta_{-t}\omega)} & \quad
\text{for } t < 0,
\end{cases}
\end{equation*}
\item[{\rm (4)}]
If \textup{(A0)(ii)}, \textup{(A1)(i)}, \textup{(A2)} and
\textup{(A3)} are satisfied, then $\tilde{\lambda}_1 =
\tilde{\lambda}_1^*$.
\end{itemize}
\end{theorem}

For $\omega \in \tilde{\Omega}^{*}_1$, define $\tilde{F}_1(\omega) :=
\{\, u \in X: \langle u, w^{*}(\omega) \rangle = 0 \,\}$.  Then
$\{\tilde{F}_1(\omega)\}_{\omega \in \tilde{\Omega}^{*}_1}$ is a
family of one\nobreakdash-\hspace{0pt}codimensional subspaces of $X$,
such that $U_{\omega}(t) \tilde{F}_1(\omega) \subset
\tilde{F}_1(\theta_{t}\omega)$ for any $\omega \in
\tilde{\Omega}^{*}_1$ and any $t \in \TT^{+}$.

Assume, for the moment, (A1)(ii)--(iii).  For $\omega \in
\tilde{\Omega}_1 \cap \tilde{\Omega}^{*}_1$, let $\hat{F}_1(\omega)$
be defined by
\begin{equation}
\label{f-hat-space-eq}
\hat F_1(\omega) =
\begin{cases}
F_{\infty}(\omega) & \text{ if (2) in Theorem \ref{Oseledets-thm}
holds with } k = 1, \\
\bigoplus\limits_{j=2}^{k} E_{j}(\omega) \oplus F_{\infty}(\omega)
& \text{ if (2) in Theorem \ref{Oseledets-thm} holds
with }  k > 1, \\
F_{1}(\omega)  & \text{ if (3) in Theorem \ref{Oseledets-thm} holds.}
\end{cases}
\end{equation}
Further, let $\hat{\lambda}_2$ be defined by
\begin{equation}
\label{lambda-hat-eq}
\hat{\lambda}_2 =
\begin{cases}
\lambda_2 & \text{ if (3) in Theorem~\ref{Oseledets-thm} holds, or if} \\
\null & \text{ (2) in Theorem~\ref{Oseledets-thm} holds with } k
> 1, \\[1ex]
\displaystyle
\lim_{\substack{t\to \infty \\ t \in \TT}} \frac{1}{t}
\ln{\norm{U_\omega(t)|_{\hat F_1(\omega)}}} & \text{ if (2) in
Theorem~\ref{Oseledets-thm} holds with } k = 1.
\end{cases}
\end{equation}

The next theorem shows the existence of generalized exponential
separation.

\begin{theorem}[Generalized exponential separation]
\label{separation-thm}
Assume \textup{(A0)(iii)},  \textup{(A1)(i)}, \textup{(A2)},
\textup{(A0)$^*$(iii)}, \textup{(A1)$^*$(i)}, \textup{(A2)$^*$}, and
\textup{(A4)}.  Then there is an invariant set $\tilde{\Omega}_0$,
$\PP(\tilde{\Omega}_0) = 1$, having the following properties.
\begin{itemize}
\item[{\rm (1)}]
The family $\{\tilde{P}(\omega)\}_{\omega \in \tilde{\Omega}_0}$
of projections associated with the invariant decomposition
$\tilde{E}_1(\omega) \oplus \tilde{F}_1(\omega) = X$ is strongly
measurable and tempered.
\item[{\rm (2)}]
$\tilde F_1(\omega)\cap X^+=\{0\}$ for any
$\omega\in\tilde\Omega_0$.
\item[{\rm (3)}]
For any $\omega \in \tilde{\Omega}_0$ and any $u \in X \setminus
\tilde{F}_1(\omega)$ \textup{(}in~particular, for any nonzero $u
\in X^{+}$\textup{)} there holds
\begin{equation*}
\lim_{\substack{t\to\infty \\ t \in \TT^{+}}} \frac{1}{t}
\ln{\norm{U_{\omega}(t)}} = \lim_{\substack{t\to\infty \\ t \in
\TT^{+}}} \frac{1}{t} \ln{\norm{U_{\omega}(t)u}} =
\tilde{\lambda}_1.
\end{equation*}
\item[{\rm (4)}]
There exist $\tilde{\sigma} \in (0, \infty]$ and
$\tilde{\lambda}_2 \in [-\infty,\infty)$, $\tilde{\lambda}_2 =
\tilde{\lambda}_1 - \tilde{\sigma}$, such that
\begin{equation*}
\lim_{\substack{t\to\infty \\ t \in \TT^{+}}} \frac{1}{t}
\ln{\frac{\norm{U_{\omega}(t)|_{\tilde{F}(\omega)}}}
{\norm{U_{\omega}(t) w(\omega)}}} = - \tilde{\sigma}
\end{equation*}
and
\begin{equation*}
\lim_{\substack{t\to\infty \\ t \in \TT^{+}}} \frac{1}{t}
\ln{\norm{U_{\omega}(t)|_{\tilde{F}_1(\omega)}}} =
\tilde{\lambda}_2
\end{equation*}
for each $\omega \in \tilde{\Omega}_0$.  Hence $\Phi$ admits a
generalized exponential separation.
\item[{\rm (5)}]
Assume moreover \textup{(A1)(ii)--(iii)} and
\textup{(A1)$^{*}$(ii)--(iii)}.  If Theorem
\ref{Oseledets-thm}(2) or (3) occurs, then $\tilde{\lambda}_2 =
\hat{\lambda}_2 < \tilde{\lambda}_1$ and $E_1(\omega) =
\tilde{E}_1(\omega)$ and $\hat{F}_1(\omega) =
\tilde{F}_1(\omega)$ for $\PP$\nobreakdash-\hspace{0pt}a.e.\
$\omega \in \Omega_{0}$.
\item[{\rm (6)}]
If \textup{(A5)} or \textup{(A5)$^*$} holds, then
$\tilde{\lambda}_1 > -\infty$. If additionally
\textup{(A1)(ii)--(iii)} and \textup{(A1)$^{*}$(ii)--(iii)} hold
then Theorem \ref{Oseledets-thm}(2) or (3) occurs.
\end{itemize}
\end{theorem}

The last theorem is about comparison of principal Lyapunov exponents.
\begin{theorem}[Monotonicity]
\label{theorem-comparison}
Let two \mlsps s $\Phi^{(1)} = ((U^{(1)}_{\omega}(t))_{\omega \in
\Omega, t \in \TT^{+}}, (\theta_{t})_{t \in \TT})$ and $\Phi^{(2)} =
((U^{(2)}_{\omega}(t))_{\omega \in \Omega, t \in \TT^{+}},
\allowbreak (\theta_{t})_{t \in \TT})$ have the property that
\begin{equation*}
U^{(1)}_{\omega}(t_0)u \le U^{(2)}_{\omega}(t_0)u
\end{equation*}
for some $t_0 \in \TT^{+} \setminus \{0\}$,
$\PP$\nobreakdash-\hspace{0pt}a.e.\ $\omega \in \Omega$ and all $u
\in X^{+}$. Assume that both $\Phi^{(1)}$ and $\Phi^{(2)}$ satisfy
\textup{(A0)(iii)},  \textup{(A1)(i)}, \textup{(A2)},
\textup{(A0)$^*$(iii)}, \textup{(A1)$^*$(i)}, \textup{(A2)$^*$}, and
\textup{(A4)}.  Then
\begin{equation*}
\tilde{\lambda}_1^{(1)} \le \tilde{\lambda}_1^{(2)},
\end{equation*}
where $\tilde{\lambda}_1^{(i)}$, $i = 1, 2$, denotes the generalized
principal Lyapunov exponent for $\Phi^{(i)}$.
\end{theorem}

\section{Preliminaries}
\label{section-preliminaries}
In this section, we present some preliminary materials for the use in
the proofs of the main results, including Birkhoff Ergodic Theorem,
Kingman Subadditive Ergodic Theorem, Hilbert projective metric in
ordered Banach spaces and basic properties, and oscillation ratio,
Birkhoff contraction ratio, and projective diameter of positive
operators in ordered Banach spaces and basic properties.

\subsection{Ergodic theorems}

In this subsection, we recall the  Birkhoff Ergodic Theorem and
Kingman Subadditive Ergodic Theorem.

\begin{theorem}[Birkhoff Ergodic Theorem]
\label{Birkhoff-thm}
$\quad$
\begin{itemize}
\item[{\rm (i)}]
\textup{(Discrete\nobreakdash-\hspace{0pt}Time Case)} Assume that
$(\OFP,(\theta_n)_{n\in\ZZ})$ is a metric
discrete\nobreakdash-\hspace{0pt}time dynamical system.  Let $f
\colon \Omega \to \RR$ be an $(\mathfrak{F},
\mathfrak{B}(\RR))$\nobreakdash-\hspace{0pt}measurable function,
with $f^{+} \in L_1(\OFP)$.  Then there exist
\begin{itemize}
\item[\textbullet]
an invariant set $\Omega_1 \subset \Omega$, with
$\PP(\Omega_1) = 1$, and
\item[\textbullet]
an invariant $(\mathfrak{F},
\mathfrak{B}(\RR))$\nobreakdash-\hspace{0pt}measurable
function $f_{\mathrm{av}}$, with $(f_{\mathrm{av}})^{+} \in
L_1(\OFP)$,
\end{itemize}
such that
\begin{equation*}
\lim\limits_{n \to \infty} \frac{1}{n}
\sum\limits_{i=0}^{n-1} f(\theta_{i}\omega) =
\lim\limits_{n \to \infty} \frac{1}{n}
\sum\limits_{i=0}^{n-1} f(\theta_{-i}\omega) =
f_{\mathrm{av}}(\omega)
\end{equation*}
for all $\omega \in \Omega_1$.  Moreover,
\begin{equation*}
\int\limits_{\Omega} f_{\mathrm{av}} \, d\PP =
\int\limits_{\Omega} f \, d\PP \in [-\infty,\infty).
\end{equation*}
If $(\OFP,(\theta_n)_{n\in\ZZ})$ is ergodic then
$f_{\mathrm{av}}$ is constantly equal to $\int_{\Omega} f \,
d\PP$.
\item[{\rm (ii)}]
\textup{(Continuous\nobreakdash-\hspace{0pt}Time Case)} Assume
that $(\OFP,(\theta_t)_{t\in\RR})$ is a metric flow.  Let $f
\colon \Omega \to \RR$ be an $(\mathfrak{F},
\mathfrak{B}(\RR))$\nobreakdash-\hspace{0pt}measurable function,
with $f^{+} \in L_1(\OFP)$.  Then there exist
\begin{itemize}
\item[\textbullet]
an invariant set $\Omega_1 \subset \Omega$, with
$\PP(\Omega_1) = 1$, and
\item[\textbullet]
an invariant $(\mathfrak{F},
\mathfrak{B}(\RR))$\nobreakdash-\hspace{0pt}measurable
function $f_{\mathrm{av}}$, with $(f_{\mathrm{av}})^{+}
\allowbreak \in \allowbreak L_1(\OFP)$,
\end{itemize}
such that
\begin{equation*}
\lim\limits_{t \to \infty} \frac{1}{t}
\int\limits_{0}^{t} f(\theta_{s}\omega) \, ds = \lim\limits_{t
\to \infty} \frac{1}{t} \int\limits_{-t}^{0} f(\theta_{s}\omega)
\, ds = f_{\mathrm{av}}(\omega)
\end{equation*}
for all $\omega \in \Omega_1$.  Moreover,
\begin{equation*}
\int\limits_{\Omega} f_{\mathrm{av}} \, d\PP =
\int\limits_{\Omega} f \, d\PP \in [-\infty,\infty).
\end{equation*}
If $(\OFP,(\theta_t)_{t\in\RR})$ is ergodic then
$f_{\mathrm{av}}$ is constantly equal to $\int_{\Omega} f \,
d\PP$.
\end{itemize}
\end{theorem}
\begin{lemma}
\label{lemma-zero}
Assume that $(\OFP, (\theta_t)_{t\in\TT})$ is an ergodic metric
dynamical system.  Then for each $f \in L_1(\OFP)$ the set of those
$\omega \in \Omega$ for which
\begin{equation*}
\lim\limits_{\substack {t \to \pm \infty \\ t \in \TT}} \frac{1}{t}
f(\theta_{t}\omega) = 0
\end{equation*}
has $\PP$\nobreakdash-\hspace{0pt}measure one.
\end{lemma}

\begin{theorem}[Kingman Subadditive Ergodic Theorem]
\label{Kingman-thm}
Assume that $(\OFP, \allowbreak (\theta_n)_{n\in\ZZ})$ is a metric
discrete\nobreakdash-\hspace{0pt}time dynamical system.  Let
$(f_n)_{n=1}^{\infty}$, $f_n \colon \Omega \to \RR$, be a sequence of
$(\mathfrak{F},
\mathfrak{B}(\RR))$\nobreakdash-\hspace{0pt}measurable functions,
with $(f_1)^{+} \in L_1(\OFP)$, such that
\begin{equation*}
f_{m+n}(\omega) \le f_{m}(\omega) + f_{n}(\theta_{m}\omega) \qquad
\text{for any }m, n \in \NN \text{ and any } \omega \in \Omega.
\end{equation*}
Then there exist
\begin{itemize}
\item
an invariant set $\Omega_1 \subset \Omega$, with $\PP(\Omega_1) =
1$, and
\item
an invariant $(\mathfrak{F},
\mathfrak{B}(\RR))$\nobreakdash-\hspace{0pt}measurable function
$f_{\mathrm{av}}$, with $(f_{\mathrm{av}})^{+} \in L_1(\OFP)$,
\end{itemize}
such that
\begin{equation*}
\lim\limits_{n \to \infty} \frac{1}{n} f_{n}(\omega) =
f_{\mathrm{av}}(\omega)
\end{equation*}
for all $\omega \in \Omega_1$.  Moreover,
\begin{equation*}
\int\limits_{\Omega} f_{\mathrm{av}} \, d\PP = \lim\limits_{n \to
\infty} \frac{1}{n} \int\limits_{\Omega} f_n \, d\PP = \inf\limits_{n
\in \NN} \frac{1}{n} \int\limits_{\Omega} f_n \, d\PP \in
[-\infty,\infty).
\end{equation*}
If $(\OFP, (\theta_n)_{n\in\ZZ})$ is ergodic then $f_{\mathrm{av}}$
is constantly equal to $\lim_{n\to\infty}(1/n)\int_{\Omega} f_n \,
d\PP$.
\end{theorem}

\subsection{Hilbert Projective Metric}

Throughout this subsection, we assume that $(X,X^{+})$ is an ordered
Banach space. We recall the concept of Hilbert projective metric and
present some basic properties.

\begin{definition}
\label{hilbert-metric-def}
\begin{itemize}
\item[{\rm (1)}]
For given $u, v \in X$, if $\{\,\underline{\alpha} \in
\RR:\underline{\alpha} v \le u\,\}$ is nonempty, define
\begin{equation*}
m(u/v) := \sup\{\,\underline{\alpha} \in \RR: \underline{\alpha}
v \le u\,\}.
\end{equation*}
If $\{\overline{\alpha} \in \RR: u \le \overline{\alpha} v\,\}$
is nonempty, define
\begin{equation*}
M(u/v) := \inf\{\,\overline{\alpha} \in \RR: u \le
\overline{\alpha} \,v\}.
\end{equation*}
\item[{\rm (2)}]
For given $u, v \in X$, if both $m(u/v)$ and $M(u/v)$ exist,
define
\begin{equation*}
\osc(u/v) := M(u/v) - m(u/v)
\,\,\,{\rm and}\,\,\, d(u,v) := \ln{\frac{M(u/v)}{m(u/v)}}.
\end{equation*}
$\osc(u/v)$ is called the {\em oscillation\/} of $u$ over $v$ and
$d(u,v)$ is called the {\em projective distance\/} between $u$
and $v$.
\end{itemize}
\end{definition}

It should be noted that for comparable $u, v \in X^{+} \setminus
\{0\}$  we have the following alternative:
\begin{itemize}
\item
either
\begin{equation*}
m(u/v) v < u < M(u/v) v,
\end{equation*}
\item
or there is $\alpha > 0$ such that $v = {\alpha}u$.
\end{itemize}

The following lemma follows easily.

\begin{lemma}
\label{lemma-projective-distance}
\begin{itemize}
\item[{\rm (1)}]
$d(u,v) = d(v,u)$ if $u, v \in X^{+} \setminus \{0\}$ and $u \sim
v$.
\item[{\rm (2)}]
$d(\lambda u,\mu v) = d(u,v)$ if $u, v \in X^{+} \setminus
\{0\}$, $u \sim v$, and $\lambda,\mu > 0$.
\item[{\rm (3)}]
$d(u,v) \le d(u,w) + d(w,v)$ if $u, v, w \in X^{+} \setminus
\{0\}$ and $u \sim v \sim w$.
\item[{\rm (4)}]
For any two $u,v \in X^{+} \setminus \{0\}$, $d(u,v) = 0$ implies
the existence of $\alpha > 0$ such that $v = {\alpha}u$.
\item[{\rm (5)}]
$\osc(\lambda u/\mu v) = \abs{\lambda} \mu^{-1} \osc(u/v)$ if
$\osc(u/v)$ exists, $u, v \in X^{+} \setminus \{0\}$, $\lambda
\in \RR \setminus \{0\}$ and $\mu > 0$.
\item[{\rm (6)}]
$\osc(u+v/w) \le \osc(u/w) + \osc(v/w)$ if $u, v, w \in X^{+}
\setminus \{0\}$ and $u \sim w$, $v \sim w$.
\end{itemize}
\end{lemma}

\begin{lemma}
\label{lemma-projective-vs-norm}
Assume that $X^{+}$ is normal.  Then for any $u, v \in X^{+}$, $u
\sim v$, with $\norm{u} = \norm{v} = 1$, there holds
\begin{equation*}
\norm{u-v} \le 3 \bigl( e^{d(u,v)}-1 \bigr).
\end{equation*}
If $(X, X^{+})$ is a Banach lattice then for any $u, v \in X^{+}$, $u
\sim v$, with $\norm{u} = \norm{v} = 1$, there holds
\begin{equation*}
\norm{u-v} \le e^{d(u,v)}-1 \bigr.
\end{equation*}
\end{lemma}
\begin{proof}
See \cite[Proposition 1.2.1]{Eve}.
\end{proof}

\begin{lemma}
\label{lemma-continuity}
Suppose that $u,v, u_k, v_k \in X^{+} \setminus \{0\}$ for $k = 1, 2,
\dots$. If  $u_k \sim v_k$ for $k = 1, 2, \dots$,
$(1/m(u_k/v_k))_{m=1}^{\infty}$ and $(M(u_k/v_k))_{m=1}^{\infty}$ are
bounded sequences, and $u_k \to u$, $v_k \to v$ as $k \to \infty$,
then $u \sim v$ and $d(u,v) \le \liminf_{k\to\infty} d(u_k,v_k)$.
\end{lemma}
\begin{proof}
First, let $M > 1$ be such that
\begin{equation*}
m(u_k/v_k) \ge \frac{1}{M-1} \quad {\rm and}\quad
M(u_k/v_k) \le M-1 \quad \text{for} \quad k = 1, 2, \dots.
\end{equation*}
Then there are $\frac{1}{M} \le \underline{\alpha}_k$,
$\overline{\alpha}_k \le M$ such that
\begin{equation*}
\underline{\alpha}_k v_k \le u_k \le \overline{\alpha}_k v_k.
\end{equation*}
By (C2), $\underline{\alpha}_k \le \overline{\alpha}_k$.  Without
loss of generality, we may assume that $\underline{\alpha}_k \to
\underline{\alpha}$ and $\overline{\alpha}_k \to \overline{\alpha}$
as $n \to \infty$. Then $0 < \underline{\alpha} \le
\overline{\alpha}$ and
\begin{equation*}
\underline{\alpha} v \le u \le \overline{\alpha} v.
\end{equation*}
Hence $u \sim v$.

For any $\epsilon > 0$ there are $\frac{1}{M} \le
\underline{\alpha}_k \le \overline{\alpha}_k \le M$ such that
\begin{equation*}
\underline{\alpha}_k v_k \le u_k \le \overline{\alpha}_k v_k
\quad\text{and}\quad d(u_k,v_k) \ge
\ln{\frac{\overline{\alpha}_k}{\underline{\alpha}_k}} - \epsilon.
\end{equation*}
Assume that $k_l \to \infty$ is such that
\begin{equation*}
\lim_{k\to\infty} d(u_{k_l},v_{k_l}) = \liminf_{k\to\infty} d(u_k,v_k)
\end{equation*}
and
\begin{equation*}
\underline{\alpha}_{k_l} \to \underline{\alpha}, \quad
\overline{\alpha}_{k_l} \to \overline{\alpha} \quad \text{as} \quad k
\to \infty.
\end{equation*}
Then
\begin{equation*}
\underline{\alpha} v \le u \le \overline{\alpha} v \quad \text{and}
\quad d(u,v) \le \ln{\frac{\overline{\alpha}}{\underline{\alpha}}}.
\end{equation*}
Therefore
\begin{equation*}
d(u,v) \le \liminf_{k\to\infty} d(u_k,v_k) + \epsilon
\end{equation*}
for any $\epsilon > 0$, and hence
\begin{equation*}
d(u,v) \le \liminf_{k\to\infty} d(u_k,v_k).
\end{equation*}
\end{proof}

\subsection{Oscillation ratio, Birkhoff contraction ratio, and
projective diameter}

Throughout this subsection, we assume that $(X,X^{+})$ is an ordered
Banach space and that $\Phi = ((U_\omega(t))_{\omega\in\Omega, t \in
\TT^{+}}, (\theta_t)_{t\in\TT})$ is a \mlsps\ on $X$ covering
$(\theta_{t})_{t \in \TT}$, satisfying (A2). At some places (A3) will
be assumed.

\begin{definition}
\label{oscillation-ratio-def}
\begin{itemize}
\item[{\rm (1)}]
For $\omega \in \Omega$ define
\begin{equation*}
p(\omega) := \sup_{\substack{u,v \in X^{+} \\ u \sim v \\ u
\ne {\alpha} v}} \frac{\osc(U_\omega(1)u/U_\omega(1)
v)}{\osc(u/v)}.
\end{equation*}
$p(\omega)$ is called the {\em oscillation ratio\/} of
$U_\omega(1)$.
\item[{\rm (2)}]
For $\omega \in \Omega$ define
\begin{equation*}
q(\omega) := \sup_{\substack{u,v \in X^{+} \\ u \sim v \\ u
\ne {\alpha} v}} \frac{d(U_\omega(1)u,U_\omega(1)v)}{d(u,v)}.
\end{equation*}
$q(\omega)$ is called the {\em Birkhoff contraction ratio\/} of
$U_\omega(1)$.
\item[{\rm (3)}]
For $\omega \in \Omega$ define
\begin{equation*}
\tau(\omega) := \sup_{\substack{u,v \in X^{+} \\
U_\omega(1)u \sim U_\omega(1) v}} d
(U_\omega(1)u,U_\omega(1)v).
\end{equation*}
$\tau(\omega)$ is called the {\em projective diameter\/} of
$U_\omega(1)$.
\end{itemize}
\end{definition}
The functions $p^{*}$, $q^{*}$ and $\tau^{*}$ for the dual $\Phi^{*}$
are defined in an analogous way.

\begin{lemma}
\label{lm-projective-diameter-nonincreasing}
For any $\omega \in \Omega$, any $t \in \TT^{+}$ and any $u, v \in
X^{+} \setminus \{0\}$ with $u \sim v$ there holds
\begin{gather*}
m(U_\omega(t)u/U_\omega(t)v) \ge m(u/v), \\
M(U_\omega(t)u/U_\omega(t)v) \le M(u/v), \\
\osc(U_\omega(t)u/U_\omega(t)v) \le \osc(u/v), \\
d(U_{\omega}(t)u, U_{\omega}(t)v) \le d(u,v).
\end{gather*}
\end{lemma}
\begin{proof}
For any $\underline{\alpha}, \overline{\alpha} > 0$ with
\begin{equation*}
\underline{\alpha} v \le u \le \overline{\alpha} v,
\end{equation*}
by (A2),
\begin{equation*}
\underline{\alpha} U_{\omega}(t)v \le U_{\omega}(t)u \le
\overline{\alpha} U_\omega(t)v \quad \forall \, \omega \in \Omega,\
t \in \TT^{+}.
\end{equation*}
This implies that
\begin{equation*}
m(U_\omega(t)u/U_\omega(t)v) \ge m(u/v) \quad \text{and} \quad
M(U_\omega(t)u/U_\omega(t)v) \le M(u/v),
\end{equation*}
thus
\begin{equation*}
\osc(U_\omega(t)u/U_\omega(t)v) \le \osc(u/v) \quad \text{and} \quad
d(U_{\omega}(t)u, U_{\omega}(t)v) \le d(u,v).
\end{equation*}
\end{proof}

The next four results will be formulated for both $\Phi$ and its dual
$\Phi^{*}$, we will however formulate their proofs for $\Phi$ only.
\begin{lemma}
\label{lm-estimate-projective-diameter}
Assume moreover \textup{(A3)} and \textup{(A3)$^*$} .  For each
$\omega \in \Omega$ and each $u \in X^{+} \setminus \{0\}$, $u^{*}
\in (X^{*})^{+} \setminus \{0\}$ there holds $U_\omega(1)u \sim
\mathbf{e}$, $U^{*}_\omega(1)u^{*} \sim \mathbf{e}^{*}$, and
\begin{equation*}
d(U_\omega(1)u, \mathbf{e}) \le \ln{\varkappa(\omega)}, \
d(U_\omega(1)u^{*}, \mathbf{e}^{*}) \le \ln{\varkappa^{*}(\omega)}.
\end{equation*}
Consequently, $\tau(\omega) \le 2 \ln{\varkappa(\omega)}$ and
$\tau^{*}(\omega) \le 2 \ln{\varkappa^{*}(\omega)}$ for any $\omega
\in \Omega$.
\end{lemma}
\begin{proof}
By (A3), for any $\omega \in \Omega$ and $u \in X^{+} \setminus
\{0\}$ we have
\begin{equation*}
m(U_\omega(1)u/\mathbf{e}) \ge \beta(\omega,u), \qquad
M(U_\omega(1)u/\mathbf{e}) \le \varkappa(\omega) \beta(\omega,u).
\end{equation*}
It suffices now to apply the definition of $d(\cdot,\cdot)$ and
Lemma~\ref{lemma-projective-distance}(3).
\end{proof}
\begin{lemma}
\label{lm-tau-finite}
Assume moreover \textup{(A3)} and \textup{(A3)$^*$}.  For each
$\omega \in \Omega$,
\begin{equation*}
\tau(\omega) < \infty, \ \tau^{*}(\omega) < \infty
\end{equation*}
and
\begin{equation*}
p(\omega) = q(\omega) = \tanh{\frac{1}{4}\tau(\omega)} (<1), \
p^{*}(\omega) = q^{*}(\omega) = \tanh{\frac{1}{4}\tau^{*}(\omega)}(<1).
\end{equation*}
\end{lemma}
\begin{proof}
By~\cite[Theorem 2.1.1]{Eve}, for any $\omega \in \Omega$, either
\begin{equation}
\label{contraction-eq1}
\tau(\omega) = \infty, \quad p(\omega) = 1, \quad q(\omega) = 1
\end{equation}
or
\begin{equation}
\label{contraction-eq2}
p(\omega) = q(\omega) = \tanh{\frac{1}{4}\tau(\omega)}.
\end{equation}
By Lemma~\ref{lm-estimate-projective-diameter}, $\tau(\omega) <
\infty$. The lemma then follows.
\end{proof}

\begin{lemma}
\label{lm-submultiplicative-1}
Assume moreover \textup{(A3)} and \textup{(A3)$^*$}.  For any $\omega
\in \Omega$, if $u, v \in X^{+} \setminus \{0\}$ are such that $u
\sim v$ but $v \ne {\alpha}u$ for any positive real $\alpha$ then
\begin{equation*}
m(U_{\omega}(1)u/U_{\omega}(1)v) > m(u/v) \quad \text{and} \quad
M(U_{\omega}(1)u/U_{\omega}(1)v) < M(u/v).
\end{equation*}
Similarly, for any $\omega \in \Omega$, if $u^{*}, v^{*} \in
(X^{*})^{+} \setminus \{0\}$ are such that $u^{*} \sim v^{*}$ but
$v^{*} \ne {\alpha}u^{*}$ for any positive real $\alpha$ then
\begin{equation*}
m(U^{*}_{\omega}(1)u^{*}/U^{*}_{\omega}(1)v^{*}) > m(u^{*}/v^{*})
\quad \text{and} \quad
M(U^{*}_{\omega}(1)u^{*}/U^{*}_{\omega}(1)v^{*}) < M(u^{*}/v^{*}).
\end{equation*}
\end{lemma}
\begin{proof}
For $u, v$ as in the assumption we have
\begin{equation*}
m(u/v) v < u < M(u/v)v.
\end{equation*}
By (A3),
\begin{equation*}
U_{\omega}(1)(u - m(u/v) v) \ge \beta(\omega, u - m(u/v) v) \mathbf{e}
\end{equation*}
and
\begin{equation*}
\mathbf{e} \ge \frac{1}{\varkappa(\omega) \beta(\omega, v)}
U_{\omega}(1)v,
\end{equation*}
which gives that
\begin{equation*}
U_{\omega}(1) u \ge \left( m(u/v) + \frac{\beta(\omega, u - m(u/v)
v)}{\varkappa(\omega) \beta(\omega, v)}\right) U_{\omega}(1)v.
\end{equation*}
The first inequality follows immediately.  The proof of the second
inequality is similar.
\end{proof}

\begin{lemma}
\label{lemma-measurability}
Let the cones $X^{+}$, $(X^{*})^{+}$ be normal and $X$, $X^*$ be
separable.  Assume moreover \textup{(A3)} and \textup{(A3)$^*$}. Then
the functions
\begin{gather*}
[\, \Omega \ni \omega \mapsto \tau(\omega) \in \RR\,],
\ [\, \Omega \ni \omega \mapsto \tau^{*}(\omega) \in \RR\,] \\
[\, \Omega \ni \omega \mapsto p(\omega) \in \RR\,], \
[\, \Omega \ni \omega \mapsto p^{*}(\omega) \in \RR\,] \\
[\, \Omega \ni \omega \mapsto q(\omega) \in \RR\,], \ [\, \Omega \ni
\omega \mapsto q^{*}(\omega) \in \RR\,]
\end{gather*}
are
$(\mathfrak{F},\mathfrak{B}(\RR))$\nobreakdash-\hspace{0pt}measurable.
\end{lemma}
\begin{proof}
By Lemma \ref{lm-tau-finite}, it suffices to prove that $[\, \Omega
\ni \omega \mapsto \tau(\omega) \in \RR\,]$ is
$(\mathfrak{F},\mathfrak{B}(\RR))$-\hspace{0pt}measurable.

First, fix $u$ and $v$ in $X^{+} \setminus\{0\}$.  We prove that $[\,
\omega \mapsto M(U_{\omega}(1)u/U_{\omega}(1)v)\,]$ is
$(\mathfrak{F},\mathfrak{B}(\RR))$\nobreakdash-\hspace{0pt}measurable.
To this end, take a countable set $\{\rho_k\}$ which is dense in
$\RR^+$. Let
\begin{equation*}
\Omega_k := \{\,\omega \in \Omega\,:\, U_{\omega}(1)u \le {\rho}_k
U_{\omega}(1)v\,\}.
\end{equation*}
Then $\Omega_k$ is the inverse image of $X^{+}$ under the function
$[\,\omega \mapsto {\rho}_k U_{\omega}(1)v - U_{\omega}(1)u \in
X\,]$. By the measurability of $U_\omega(1)$ in $\omega$, $\Omega_k$
is a measurable subset of $\Omega$. Let
\begin{equation*}
M_k(\omega) := \begin{cases} \rho_k \quad \text{for} \quad \omega \in
\Omega_k \\
\infty \quad \text{for} \quad \omega \in \Omega \setminus \Omega_k.
\end{cases}
\end{equation*}
Then $[\, \Omega \ni \omega \mapsto M_k(\omega) \in \RR \,]$ is
$(\mathfrak{F},\mathfrak{B}(\RR))$\nobreakdash-\hspace{0pt}measurable
and
\begin{equation*}
M(U_{\omega}(1)u/U_{\omega}(1)v) = \inf_{k\ge 1}M_k(\omega).
\end{equation*}
It then follows that $[\, \omega \mapsto
M(U_{\omega}(1)u/U_{\omega}(1)v)\,]$ is
$(\mathfrak{F},\mathfrak{B}(\RR))$\nobreakdash-\hspace{0pt}measurable.

Similarly, it can be proved that $[\, \omega \mapsto
m(U_{\omega}(1)u/U_{\omega}(1)v) \,]$ is
$(\mathfrak{F},\mathfrak{B}(\RR))$\nobreakdash-\hspace{0pt}measurable,
for fixed $u$ and $v$ in $X^{+} \setminus\{0\}$.

Hence $[\, \omega \mapsto d(U_\omega(1)u,U_\omega(1)v) \,]$ is
$(\mathfrak{F},\mathfrak{B}(\RR))$\nobreakdash-\hspace{0pt}measurable,
for fixed $u$ and $v$ in $X^{+} \setminus \{0\}$.

Now let $\{u_k\}$ and $\{v_l\}$ be two dense countable sets in $X^{+}
\setminus \{0\}$.  We claim that
\begin{equation}
\label{aux-eq1}
\tau(\omega) = \sup_{k,l\ge 1} d(U_\omega(1)u_k,U_\omega(1)v_l)
\end{equation}
and hence $\tau(\omega)$ is measurable in $\omega$.  In~fact, for any
$u, v \in X^+ \setminus \{0\}$ there are $(u_{k_m})_{m=1}^{\infty}$
and $(v_{l_m})_{k=1}^{\infty}$ such that $u_{k_m} \to u$ and $v_{l_m}
\to v$ as $m \to \infty$.  Observe that it follows from (A3) that
\begin{equation}
\label{m-and-M}
M(U_{\omega}(1)u_{k_m}/U_{\omega}(1)v_{l_m}) \le \varkappa^2(\omega)
\cdot m(U_{\omega}(1)u_{k_m}/U_{\omega}(1)v_{l_m})
\end{equation}
for $m = 1, 2, \dots$.

We prove that $\{\, m(U_{\omega}(1)u_{k_m}/U_{\omega}(1)v_{k_m}): m =
1, 2, \dots \,\}$ is bounded away from zero.  Indeed, if not then
there is a subsequence $m_j \to \infty$ as $j \to \infty$ such that
$m(U_{\omega}(1)u_{k_{m_j}}/U_{\omega}(1)v_{l_{m_j}}) \to 0$ as $j
\to \infty$, from which it follows via~\eqref{m-and-M} that
$U_{\omega}(1)v_{l_{m_j}} \to 0$, which contradicts the fact that
$U_{\omega}(1)v_{l_{m_j}} \to U_{\omega}(1)v \ne 0$.  Now we prove
that $\{\, M(U_{\omega}(1)u_{k_m}/U_{\omega}(1)v_{l_m}): m = 1, 2,
\dots \,\}$ is bounded.  Indeed, if not then there is a subsequence
$m_j \to \infty$ as $j \to \infty$ such that
$M(U_{\omega}(1)u_{k_{m_j}}/U_{\omega}(1)v_{l_{m_j}}) \to \infty$ as
$j \to \infty$, from which it follows via~\eqref{m-and-M} that
$m(U_{\omega}(1)u_{k_{m_j}}/U_{\omega}(1)v_{l_{m_j}}) \to \infty$ as
$j \to \infty$.  It follows via the normality of the cone $X^{+}$
that $\norm{U_{\omega}(1)u_{k_{m_j}}} \to \infty$, which contradicts
the fact that $U_{\omega}(1)u_{k_{m_j}} \to U_{\omega}(1)u$.

Hence $\{M(U_{\omega}(1)u_{k_m}/U_{\omega}(1)v_{l_m})\}$ and
$\{\frac{1}{m(U_{\omega}(1)u_{k_m}/U_{\omega}(1)v_{l_m})}\}$ are
bounded.  By Lemma \ref{lemma-continuity},
\begin{equation*}
d(U_\omega(1)u,U_\omega(1)v) \le \liminf_{m\to\infty}
d(U_\omega(1)u_{k_m},U_\omega(1)v_{l_m}).
\end{equation*}
This implies that \eqref{aux-eq1} holds and then $[\, \omega \mapsto
\tau(\omega) \,]$ is
$(\mathfrak{F},\mathfrak{B}(\RR))$\nobreakdash-\hspace{0pt}measurable.
\end{proof}

\section{Proofs of Main Results}
\label{section-proofs}

Throughout the entire Section~\ref{section-proofs} we assume (A0)(i)
and that $\Phi = ((U_\omega(t))_{\omega\in\Omega, t \in \TT^{+}},
\allowbreak (\theta_t)_{t\in\TT} )$ is a \mlsps\ on $X$ covering
$(\theta_{t})_{t \in \TT}$, satisfying (A1)(i) and (A2).

\smallskip
For a closed $E \subset X^{+}$, $E \ne \{0\}$, such that $u \in E$
and $\alpha \ge 0$ implies ${\alpha}u \in E$, the symbol
$\mathcal{S}_1(E)$ will denote the intersection of $E$ with the unit
sphere in $X$, $\mathcal{S}_1(E) := \{\, u \in E: \norm{u} = 1\,\}$.

\smallskip
Under the assumption (A1)(ii) or (A3) we define, for $t \in \TT^{+}$
and $\omega \in \Omega$, a function $\mathcal{U}_{\omega}(t) \colon
\mathcal{S}_1(X^{+})  \to \mathcal{S}_1(X^{+})$ by the formula:
\begin{equation*}
\mathcal{U}_{\omega}(t)u := \frac{U_{\omega}(t)u}
{\norm{U_{\omega}(t)u}}, \quad u \in \mathcal{S}_1(X^{+}) .
\end{equation*}
The function $\mathcal{U}_{\omega}(t)$ is well defined.  This is
obvious under (A1)(ii).  Assume then (A3).  If $U_{\omega}(t)u = 0$
for some $t \in \TT^{+}$ and $u \in \mathcal{S}_1(X^{+})$ then there
exists $n_0 \in \NN \cup \{0\}$ such that $U_{\omega}(n_0)u \in X^{+}
\setminus \{0\}$ but $U_{\omega}(n_0+1)u = 0$.  As
$U_{\omega}(n_0+1)u = U_{\theta_{n_0}\omega}(1)(U_{\omega}(n_0)u)$,
this contradicts (A3).

The function $\mathcal{U}_{\omega}(t)$ is clearly continuous.
Furthermore, as a consequence of~\eqref{eq-cocycle} we have
\begin{equation}
\label{eq-cocycle-projective}
\mathcal{U}_{\omega}(s+t) = \mathcal{U}_{\theta_{s}\omega}(t) \circ
\mathcal{U}_{\omega}(s), \quad s, t \in \TT^{+},\ \omega \in \Omega.
\end{equation}

\subsection{Entire positive orbits and proof of Theorem
\ref{thm-largest-meets-nontrivially}}
\label{subsection-entire}

Throughout this subsection we assume moreover (A1)(ii)--(iii) and
that $X^{+}$ is total.  Further, we assume that Theorem
\ref{Oseledets-thm}(2) or (3) occurs.  We investigate the existence
of entire positive orbits of $\Phi$ and prove Theorem
\ref{thm-largest-meets-nontrivially}.

Let $\Omega_0$ be as in Theorem \ref{Oseledets-thm}.  Recall that
$\{E_1(\omega)\}_{\omega \in \Omega_0}$ is a measurable family of
finite\nobreakdash-\hspace{0pt}dimensional vector subspaces of $X$
such that for each $\omega \in \Omega_0$ and each $t \in \TT^{+}$ the
mapping $U_\omega(t)|_{E_1(\omega)} \colon E_1(\omega) \to
E_1(\theta_{t}\omega)$ is a linear isomorphism.

\smallskip
Let $\hat{F}_1(\omega)$ be as in \eqref{f-hat-space-eq}. Let
$\{P(\omega)\}_{\omega \in \Omega_0}$ be the family of projections
associated with the decomposition $E_1(\omega) \oplus
\hat{F}_1(\omega) = X$. For $\omega \in \Omega_0$ put
\begin{equation*}
\delta(\omega) := \inf{\left\{\, \frac{\norm{P(\omega)u}}{\norm{u -
P(\omega)u}}: u \in X^{+} \setminus \{0\} \,\right\}}.
\end{equation*}
Observe that for any nonzero $u \in X^{+}$, $P(\omega)u$ and $u -
P(\omega)u$ cannot be both zero, so $\norm{P(\omega)u}/\norm{u -
P(\omega)u}$ is well defined (perhaps equal to $\infty$).

We claim that if $X^+$ is total, then $\delta(\omega)$ is a
nonnegative real number for each $\omega \in \Omega_0$.  Indeed,
$\delta(\omega) = \infty$ for some $\omega \in \Omega_0$ means that
$X^{+} \subset \hat{F}_1(\omega)$. As the cone $X^{+}$ is total, we
have that $X = \cl{(X^{+} - X^{+})} \subset \hat{F}_1(\omega)$, which
is impossible.

\begin{lemma}
\label{auxiliary-lemma-1}
For any $\omega \in \Omega_0$, $\delta(\omega) = 0$ if~and only~if
$E_1(\omega) \cap X^{+} \varsupsetneq \{0\}$.
\end{lemma}
\begin{proof}
The ``if'' part is straightforward.  Assume that $\delta(\omega) = 0$
for some $\omega \in \Omega_0$.  It follows that for each $k = 1, 2,
\dots$ we can choose $u_k \in X^{+}$, $\norm{u_k} = 1$, such that
$\norm{P(\omega) u_k} < \frac{1}{k} \norm{u_k - P(\omega) u_k}$.
Since the set $\{\, (\Id_{X}-P(\omega)) u_k: k \in \NN \,\}$, being a
bounded subset of a finite\nobreakdash-\hspace{0pt}dimensional vector
subspace $E_1(\omega)$, has compact closure, we can extract a
subsequence $((\Id_{X}-P(\omega)) u_{k_l})_{l=1}^{\infty}$ convergent
to some $v \in E_1(\omega)$. Observe that $P(\omega) u_{k_l} \to 0$,
consequently $u_{k_l} \to v$ as $l \to \infty$.  We have $v \in
X^{+}$ and $\norm{v} = 1$.
\end{proof}

\begin{lemma}
\label{auxiliary-lemma-2}
The function $\delta \colon \Omega_0 \to [0,\infty)$ is
$(\mathfrak{F},
\mathfrak{B}(\RR))$\nobreakdash-\hspace{0pt}measurable.
\end{lemma}
\begin{proof}
Take $\{z_k\}_{k=1}^{\infty}$ to be a countable dense subset of
$\mathcal{S}_1(X^{+})$.  Notice that for any $\omega \in \Omega_0$
and any positive real $r$, ``$\delta(\omega) \ge r$'' is equivalent
to ``$\norm{P(\omega) z_k}/\allowbreak \norm{z_k - P(\omega) z_k} \ge
r$ for all $k = 1, 2 ,3, \dots$.''  As $\{P(\omega)\}_{\omega \in
\Omega_0}$ is strongly measurable, for any positive real $r$ and any
$k = 1, 2 , 3, \dots$ the set $\{\, \omega \in \Omega_0:
\norm{P(\omega) z_k}/\allowbreak \norm{z_k - P(\omega) z_k} \ge
r\,\}$ belongs to $\mathfrak{F}$. This implies that  $\delta \colon
\Omega_0 \to [0,\infty)$ is $(\mathfrak{F},
\mathfrak{B}(\RR))$\nobreakdash-\hspace{0pt}measurable.
\end{proof}

\begin{proof}[Proof of Theorem~\ref{thm-largest-meets-nontrivially}]
We first prove that $E_1(\omega) \cap X^{+} \varsupsetneq \{0\}$ for
a.e. $\omega\in\Omega$.

Fix two real numbers $\underline{\lambda} < \overline{\lambda}$ in
the following way.  If (2) in Theorem~\ref{Oseledets-thm} holds with
$k = 1$ we stipulate only that $\underline{\lambda} <
\overline{\lambda} < \lambda_1$.  Otherwise we take $\lambda_2 <
\underline{\lambda} < \overline{\lambda} < \lambda_1$

For each $\omega \in \Omega_0$ there are $\overline{c}(\omega) \in
(0,1]$ and $\underline{c}(\omega) \ge 1$ such that
\begin{equation}
\label{c-1}
\begin{aligned}
\norm{U_{\omega}(t)u} & \ge \overline{c}(\omega)
e^{\overline{\lambda} t} \norm{u} & \quad &
\text{for any }u \in E_1(\omega)\text{ and }t \in \TT^{+} \\
\norm{U_{\omega}(t)u} & \le \underline{c}(\omega)
e^{\underline{\lambda} t} \norm{u} & \quad & \text{for any }u \in
\hat{F}_1(\omega)\text{ and }t \in \TT^{+}.
\end{aligned}
\end{equation}
Consequently,
\begin{align*}
\frac{\norm{P(\theta_{t}\omega) U_{\omega}(t)u}}{\norm{(\Id_{X} -
P(\theta_{t}\omega)) U_{\omega}(t)u}} & = \frac{\norm{U_{\omega}(t)
P(\omega) u}}{\norm{U_{\omega}(t) (\Id_{X} - P(\omega)) u}} \\[0.5em]
& \le \frac{\underline{c}(\omega)}{\overline{c}(\omega)} \,
e^{(\underline{\lambda}-\overline{\lambda})t}
\frac{\norm{P(\omega)u}}{\norm{u - P(\omega)u}}
\end{align*}
for each $u \in X \setminus \hat{F}_1(\omega)$ and each $t \in
\TT^{+}$. Therefore we have
\begin{equation*}
\delta(\theta_{t}\omega) \le
\frac{\underline{c}(\omega)}{\overline{c}(\omega)} \,
e^{(\underline{\lambda}-\overline{\lambda})t} \, \delta(\omega)
\end{equation*}
for all $\omega \in \Omega_{0}$ and $t \in \TT^{+}$, which implies
that
\begin{equation*}
\lim\limits_{n\to\infty} \frac{1}{n} \sum\limits_{i=0}^{n-1}
\delta(\theta_{i}\omega) = 0
\end{equation*}
for all $\omega \in \Omega_{0}$.  We apply the Birkhoff Ergodic
Theorem (Theorem~\ref{Birkhoff-thm}(i)) to $((U_{\omega}(n))_{\omega
\in \Omega, n \in \ZZ^{+}}$, $ (\theta_n)_{n \in \ZZ})$ and the
function $-\delta$ to conclude that $\int_{\Omega} \delta \, d\PP =
0$, from which it follows that $\delta(\omega) = 0$ for $\omega$
belonging to $\Omega_1$ with $\PP(\Omega_1) = 1$ and
$\theta_{1}(\Omega_1) = \Omega_1$.  An application of
Lemma~\ref{auxiliary-lemma-1} gives that $E_1(\theta_{n}\omega) \cap
X^{+} \varsupsetneq \{0\}$ for all $\omega \in \Omega_1$ and all $n
\in \ZZ$.  This finishes the proof in the
discrete\nobreakdash-\hspace{0pt}time case.  In the
continuous\nobreakdash-\hspace{0pt}time case, let $t \in \RR
\setminus \ZZ$.  Pick a nonzero $u \in
E_1(\theta_{\intpart{t}}(\omega)) \cap X^{+}$.  We have a nonzero
$U_{\theta_{\intpart{t}}\omega}(t - \intpart{t})u \in
E_{1}(\theta_{t}\omega) \cap X^{+}$.

Next we prove that for each $\omega \in \Omega_1$ there exists an
entire positive  orbit $v_{\omega} \colon \TT \to X^+$ of
$U_{\omega}$ such that
\begin{equation*}
v_{\omega}(t) \in (E_1(\theta_{t}\omega) \cap X^{+}) \setminus
\{0\} \quad \forall t \in \TT.
\end{equation*}
Fix $\omega \in \Omega_1$.  For $n = 1,2,3, \dots$ the sets
$\mathcal{U}_{\theta_{-n}\omega}(n)(\mathcal{S}_1(E_1(\theta_{-n}(\omega))
\cap X^{+}))$ are compact and nonempty.  Further, it follows
from~\eqref{eq-cocycle-projective} that they form a nonincreasing
family, consequently their intersection, $G_0$, is a nonempty compact
set.  It suffices now to pick one $u \in G_0$ and put $v_{\omega}(t)
:= U_{\omega}(t)u$, $t \in \TT$, where $U_{\omega}(t)$ is, for $t <
0$, understood as
$(U_{\theta_{-t}\omega}(-t)|_{E_1(\theta_{-t}\omega)})^{-1}$.
\end{proof}

\subsection{Principal Floquet subspaces and proofs of Theorems
\ref{theorem-w} and \ref{theorem-w-star}}
\label{section-principal}

In this subsection, we investigate the existence of generalized
principal Floquet subspaces and principal Lyapunov exponents and
prove Theorems \ref{theorem-w} and \ref{theorem-w-star}. Throughout
this subsection, we assume additionally (A0)(ii) and (A3).

Before proving Theorems \ref{theorem-w} and \ref{theorem-w-star}, we
first prove some propositions.

\begin{proposition}
\label{prop-Cauchy-estimate}
\begin{itemize}
\item[{\rm (i)}]
Let $\omega \in \Omega$, $t \in \TT^{+}$, $2 \le t \le t_1 $ and
$u, \tilde{u} \in \mathcal{S}_1(X^{+})$. Then
\begin{equation}
\label{Cauchy-estimate-02}
\lVert \mathcal{U}_{\theta_{-t}\omega}(t_1) u -
\mathcal{U}_{\theta_{-t}\omega}(t_1) \tilde{u} \rVert \le 6
\varkappa^2(\theta_{-1}\omega) \,
\big(\ln{\varkappa(\theta_{-\intpart{t}}\omega)}\big) \,
q(\theta_{-\intpart{t}+1}\omega) \dots q(\theta_{-1}\omega).
\end{equation}
In~particular, for $n = 2, 3, \dots$ one has
\begin{equation}
\label{Cauchy-estimate-020}
\lVert \mathcal{U}_{\theta_{-n}\omega}(n) u -
\mathcal{U}_{\theta_{-n}\omega}(n) \tilde{u} \rVert \le 6
\varkappa^2(\theta_{-1}\omega) \,
\big(\ln{\varkappa(\theta_{-n}\omega)}\big) \,
q(\theta_{-n+1}\omega) \dots q(\theta_{-1}\omega).
\end{equation}
\item[{\rm (ii)}]
Let $\omega \in \Omega$, $t \in \TT^{+}$, $t \ge 2$ and $u,
\tilde{u} \in \mathcal{S}_1(X^{+})$.  Then
\begin{equation}
\label{Cauchy-estimate-03}
\lVert \mathcal{U}_{\omega}(t) u - \mathcal{U}_{\omega}(t)
\tilde{u} \rVert \le 6 \varkappa^2(\theta_{\intpart{t}-1}\omega)
\, \big(\ln{\varkappa(\omega)}\big) \, q(\theta_{1}\omega) \dots
q(\theta_{\intpart{t}-1}\omega).
\end{equation}
In~particular, for $n = 2, 3, \dots$ one has
\begin{equation}
\label{Cauchy-estimate-030}
\lVert \mathcal{U}_{\omega}(n) u - \mathcal{U}_{\omega}(n)
\tilde{u} \rVert \le 6 \varkappa^2(\theta_{n-1}\omega) \,
\big(\ln{\varkappa(\omega)} \big)\, q(\theta_{1}\omega) \dots
q(\theta_{n-1}\omega).
\end{equation}
\end{itemize}
\end{proposition}
\begin{proof}
(i)  Observe that, by~\eqref{eq-cocycle-projective}, we have
\begin{equation*}
\mathcal{U}_{\theta_{-t}\omega}(t) =
\mathcal{U}_{\theta_{-1}\omega}(1) \circ \dots \circ
\mathcal{U}_{\theta_{-\intpart{t}+1}\omega}(1) \circ
\mathcal{U}_{\theta_{-t}\omega}(t - \intpart{t} +1),
\end{equation*}
consequently, by the definition of $q$,
\begin{multline}
\label{Cauchy-estimate-1}
d(\mathcal{U}_{\theta_{-t}\omega}(t)u,
\mathcal{U}_{\theta_{-t}\omega}(t)\tilde{u}) \\
\le q(\theta_{-\intpart{t}+1}\omega) \dots q(\theta_{-1}\omega) \,
d(\mathcal{U}_{\theta_{-t}\omega}(t-\intpart{t}+1)u,
\mathcal{U}_{\theta_{-t}\omega}(t-\intpart{t}+1)\tilde{u}).
\end{multline}
Since both $\mathcal{U}_{\theta_{-t}\omega}(t-\intpart{t}+1)u$ and
$\mathcal{U}_{\theta_{-t}\omega}(t-\intpart{t}+1)\tilde{u}$ belong to
the image of $\mathcal{S}_1(X^{+})$ under
$\mathcal{U}_{\theta_{-\intpart{t}}\omega}(1)$, it follows from
Lemma~\ref{lm-estimate-projective-diameter} that
\begin{equation}
\label{Cauchy-estimate-2}
d(\mathcal{U}_{\theta_{-t}\omega}(t-\intpart{t}+1)u,
\mathcal{U}_{\theta_{-t}\omega}(t-\intpart{t}+1)\tilde{u}) \le 2
\ln{\varkappa(\theta_{-\intpart{t}}\omega)}.
\end{equation}
As both $\mathcal{U}_{\theta_{-t}\omega}(t)u$ and
$\mathcal{U}_{\theta_{-t}\omega}(t)\tilde{u}$ belong to the image of
$\mathcal{S}_1(X^{+})$ under $\mathcal{U}_{\theta_{-1}\omega}(1)$,
there holds
\begin{equation*}
d(\mathcal{U}_{\theta_{-t}\omega}(t)u,
\mathcal{U}_{\theta_{-t}\omega}(t)\tilde{u}) \le
\tau(\theta_{-1}\omega).
\end{equation*}
By Lemma~\ref{lemma-projective-vs-norm},
\begin{equation*}
\lVert \mathcal{U}_{\theta_{-t}\omega}(t) u -
\mathcal{U}_{\theta_{-t}\omega}(t) \tilde{u} \rVert \le 3 (
\exp{d(\mathcal{U}_{\theta_{-t}\omega}(t) u,
\mathcal{U}_{\theta_{-t}\omega}(t) \tilde{u})} - 1),
\end{equation*}
which is, by standard calculus, $\le 3
d(\mathcal{U}_{\theta_{-t}\omega}(t)u,
\mathcal{U}_{\theta_{-t}\omega}(t)\tilde{u}) \,
\exp{\tau(\theta_{-1}\omega)}$.  Putting the above inequalities
together and applying
Lemmas~\ref{lm-projective-diameter-nonincreasing} and
\ref{lm-estimate-projective-diameter} we obtain
\begin{equation}
\label{Cauchy-estimate-3}
\begin{aligned}
\lVert \mathcal{U}_{\theta_{-t}\omega}(t_1) u -
\mathcal{U}_{\theta_{-t}\omega}(t_1) \tilde{u} \rVert & \le 3
d(\mathcal{U}_{\theta_{-t}\omega}(t_1)u,
\mathcal{U}_{\theta_{-t}\omega}(t_1)\tilde{u}) \,
\exp{\tau(\theta_{-1}\omega)} \\
& \le 3 d(\mathcal{U}_{\theta_{-t}\omega}(t)u,
\mathcal{U}_{\theta_{-t}\omega}(t)\tilde{u}) \,
\exp{\tau(\theta_{-1}\omega)} & \qquad \text{by
Lemma~\ref{lm-projective-diameter-nonincreasing}} \\
& \le 3 d(\mathcal{U}_{\theta_{-t}\omega}(t)u,
\mathcal{U}_{\theta_{-t}\omega}(t)\tilde{u}) \,
\varkappa^2(\theta_{-1}\omega). & \qquad \text{by
Lemma~\ref{lm-estimate-projective-diameter}}
\end{aligned}
\end{equation}
\eqref{Cauchy-estimate-1}--\eqref{Cauchy-estimate-3}
give~\eqref{Cauchy-estimate-02}.

(ii)  Observe that we have
\begin{equation}
\label{Cauchy-estimate-11}
\begin{aligned}
d(\mathcal{U}_{\omega}(\intpart{t})u, {} &
\mathcal{U}_{\omega}(\intpart{t})\tilde{u})
\\
& \le q(\theta_{1}\omega) \dots q(\theta_{\intpart{t}-1}\omega)
\,d(\mathcal{U}_{\omega}(1)u, \mathcal{U}_{\omega}(1)\tilde{u}) &
\quad \text{by
the definition of $q$} \\
& \le \tau(\omega) \, q(\theta_{1}\omega) \dots
q(\theta_{\intpart{t}-1}\omega) \\
& \le 2 \ln{\varkappa(\omega)} \, q(\theta_{1}\omega) \dots
q(\theta_{\intpart{t}-1}\omega). & \qquad \text{by
Lemma~\ref{lm-estimate-projective-diameter}}
\end{aligned}
\end{equation}
As both $\mathcal{U}_{\omega}(\intpart{t})u$ and
$\mathcal{U}_{\omega}(\intpart{t})\tilde{u}$ belong to the image of
$\mathcal{S}_1(X^{+})$ under
$\mathcal{U}_{\theta_{\intpart{t}-1}\omega}(1)$, we obtain, applying
Lemma~~\ref{lm-projective-diameter-nonincreasing}, that
\begin{equation*}
d(\mathcal{U}_{\omega}(t)u, \mathcal{U}_{\omega}(t)\tilde{u}) \le
d(\mathcal{U}_{\omega}(\intpart{t})u,
\mathcal{U}_{\omega}(\intpart{t})\tilde{u}) \le
\tau(\theta_{\intpart{t}-1}\omega).
\end{equation*}
By Lemma~\ref{lemma-projective-vs-norm},
\begin{equation*}
\lVert \mathcal{U}_{\omega}(t) u - \mathcal{U}_{\omega}(t) \tilde{u}
\rVert \le 3 ( \exp{d(\mathcal{U}_{\omega}(t) u,
\mathcal{U}_{\omega}(t) \tilde{u})} - 1),
\end{equation*}
which is, by standard calculus, $\le 3 d(\mathcal{U}_{\omega}(t)u,
\mathcal{U}_{\omega}(t)\tilde{u}) \,
\exp{\tau(\theta_{\intpart{t}-1}\omega)}$.  An application of
Lemma~\ref{lm-estimate-projective-diameter} yields
\begin{equation}
\label{Cauchy-estimate-12}
\lVert \mathcal{U}_{\omega}(t) u - \mathcal{U}_{\omega}(t) \tilde{u}
\rVert \le 3 d(\mathcal{U}_{\omega}(t)u,
\mathcal{U}_{\omega}(t)\tilde{u}) \,
\varkappa^2(\theta_{\intpart{t}-1}\omega).
\end{equation}
\eqref{Cauchy-estimate-11} and \eqref{Cauchy-estimate-12} give
\eqref{Cauchy-estimate-03}.
\end{proof}

\begin{proposition}
\label{prop-auxiliary-q}
Let $I := \int_{\Omega} \ln{q} \, d\PP$.  Then there exists an
invariant $\bar{\Omega}_1 \subset \Omega$, $\PP(\bar{\Omega}_1) = 1$,
with the property that
\begin{enumerate}
\item[{\rm (1)}]
for any $I < J < 0$ and any $\omega \in \bar{\Omega}_1$ there is
$C_1(J, \omega) > 0$ such that
\begin{equation*}
\lVert \mathcal{U}_{\theta_{-t}\omega}(t) u -
\mathcal{U}_{\theta_{-t}\omega}(t) \tilde{u} \rVert \le C_1(J,
\omega) e^{Jt}
\end{equation*}
for any $t \in \TT^{+}$, $t \ge 3$, and any $u, \tilde{u} \in
\mathcal{S}_1(X^{+})$.
\item[{\rm (2)}]
for any $I < J < 0$ and any $\omega \in \bar{\Omega}_1$ there is
$C_2(J, \omega) > 0$ such that
\begin{equation*}
\lVert \mathcal{U}_{\omega}(t) u - \mathcal{U}_{\omega}(t)
\tilde{u} \rVert \le C_2(J, \omega) e^{Jt}
\end{equation*}
for any $t \in \TT^{+}$, $t \ge 2$, and any $u, \tilde{u} \in
\mathcal{S}_1(X^{+})$.
\end{enumerate}
\end{proposition}

\begin{proof}
(1)  Consider first the discrete\nobreakdash-\hspace{0pt}time case.

It follows from the Birkhoff Ergodic Theorem
(Theorem~\ref{Birkhoff-thm}(i)) applied to the $(\mathfrak{F},
\mathfrak{B}(\RR))$\nobreakdash-\hspace{0pt}measurable function
$\ln{q} \colon \Omega \to (-\infty,0)$ that the invariant set
$\Omega'$ consisting of those $\omega \in \Omega$ for which
\begin{equation}
\label{Birkhoff-1}
\lim_{n \to \infty} \frac{1}{n} \sum_{i = 1}^{n}
\ln{q(\theta_{-i}\omega)} = I
\end{equation}
has $\PP$\nobreakdash-\hspace{0pt}measure one.  Since
$\ln^+\ln{\varkappa} \in L_1(\OFP)$, Lemma~\ref{lemma-zero}
establishes the existence of an invariant $\Omega'' \subset \Omega$,
$\PP(\Omega'') = 1$, such that
\begin{equation}
\label{Birkhoff-3}
\limsup_{n \to \infty} \frac{\ln{\ln{\varkappa(\theta_{-n}\omega)}}}{n}
\le 0 \qquad \text{for all } \omega \in \Omega''.
\end{equation}
Let $\bar{\Omega}_1^{(1)} := \Omega' \cap \Omega''$.  The set
$\bar{\Omega}_1^{(1)}$ is invariant, with
$\PP(\tilde{\Omega}_1^{(1)}) = 1$.  By~\eqref{Cauchy-estimate-02},
\eqref{Birkhoff-1} and~\eqref{Birkhoff-3},
\begin{equation*}
\limsup\limits_{n \to \infty} \sup\limits_{u, \tilde{u} \in
\mathcal{S}_1(X^{+})}
\frac{\ln{\lVert \mathcal{U}_{\theta_{-n}\omega}(n) u -
\mathcal{U}_{\theta_{-n}\omega}(n) \tilde{u} \rVert}}{n} \le I
\end{equation*}
for each $\omega \in \bar{\Omega}_1^{(1)}$.  Therefore, for any $J
\in (I, 0)$ and any $\omega \in \bar{\Omega}_1^{(1)}$ there is $N =
N(J, \omega) \in \NN$ such that
\begin{equation*}
\lVert \mathcal{U}_{\theta_{-n}\omega}(n) u -
\mathcal{U}_{\theta_{-n}\omega}(n) \tilde{u} \rVert \le Jn
\end{equation*}
for all $n = N, N+1, \dots$ and any $u, \tilde{u} \in
\mathcal{S}_1(X^{+})$.  It suffices now to apply the
estimate~\eqref{Cauchy-estimate-020} to $n = 3, 4, \dots, N-1$ to get
the desired result.

\smallskip
We proceed now to the continuous\nobreakdash-\hspace{0pt}time case.

It follows from the Birkhoff Ergodic Theorem
(Theorem~\ref{Birkhoff-thm}(i)) applied to the
discrete\nobreakdash-\hspace{0pt}time metric dynamical system $(\OFP,
(\theta_{n})_{n \in \ZZ})$ and the $(\mathfrak{F},
\mathfrak{B}(\RR))$\nobreakdash-\hspace{0pt}measurable function
$\ln{q} \colon \Omega \to (-\infty,0)$ that the set $\Omega'$ of
those $\omega \in \Omega$ for which the limit
\begin{equation}
\label{Birkhoff-1-cont}
\lim_{n \to \infty} \frac{1}{n} \sum_{i =
1}^{n} \ln{q(\theta_{-i}\omega)} =: (\ln{q})_{\mathrm{av}}(\omega)
\end{equation}
exists has $\PP$\nobreakdash-\hspace{0pt}measure one.  Since
$\ln^+\ln{\varkappa} \in L_1(\OFP)$ (see (A3)),
Lemma~\ref{lemma-zero} establishes the existence of $\Omega'' \subset
\Omega$, $\theta_{1}(\Omega') = \Omega'$, $\PP(\Omega'') = 1$, such
that
\begin{equation}
\label{Birkhoff-3-cont}
\limsup_{n \to \infty} \frac{\ln{\ln{\varkappa(\theta_{-n}\omega)}}}{n}
\le 0 \qquad \text{for all } \omega \in \Omega''.
\end{equation}

Let $\bar{\Omega}_1^{(1)} := \bigcup\limits_{0\leq T\leq
1}\theta_{T}(\Omega' \cap \Omega'')$.  The set $\bar{\Omega}_1^{(1)}$
is invariant, and contains the set $\Omega' \cap \Omega''$ of full
$\PP$\nobreakdash-\hspace{0pt}measure.  Since $\PP$ is complete,
$\bar{\Omega}_1^{(1)} \in \mathfrak{F}$ and
$\PP(\bar{\Omega}_1^{(1)}) = 1$.  By~\eqref{Cauchy-estimate-02}, we
estimate, for any $\omega \in \bar{\Omega}_1^{(1)}$ of the form
$\theta_{T}\tilde{\omega}$, $T \in [0,1)$, $\tilde{\omega} \in
\Omega' \cap \Omega''$,
\begin{equation}
\label{Birkhoff-11-cont}
\begin{aligned}
\lVert \mathcal{U}_{\theta_{-t}\omega}(t) u -
\mathcal{U}_{\theta_{-t}\omega}(t) \tilde{u} \lVert & = \lVert
\mathcal{U}_{\theta_{-t + T}\tilde\omega}(t) u -
\mathcal{U}_{\theta_{-t + T}\tilde\omega}(t) \tilde{u} \rVert
\\
&
\le 6 \varkappa^2(\theta_{-1}\tilde{\omega}) \,
\ln{\varkappa(\theta_{-\intpart{t-T}}\tilde{\omega})} \,
q(\theta_{-\intpart{t-T}+1}\tilde{\omega}) \dots
q(\theta_{-1}\tilde{\omega}),
\end{aligned}
\end{equation}
where $t \ge 3$ and $u, \tilde{u} \in \mathcal{S}_1(X^{+})$.

It follows from~\eqref{Birkhoff-1-cont}, \eqref{Birkhoff-3-cont}
and~\eqref{Birkhoff-11-cont} that
\begin{equation*}
\limsup\limits_{t \to \infty} \sup\limits_{u, \tilde{u} \in
\mathcal{S}_1(X^{+})} \frac{\ln{\lVert
\mathcal{U}_{\theta_{-t}\omega}(t) u -
\mathcal{U}_{\theta_{-t}\omega}(t) \tilde{u} \rVert}}{t} \le
(\ln{q})_{\mathrm{av}}(\tilde{\omega})
\end{equation*}
for all $\omega \in \bar{\Omega}_1^{(1)}$.  Since
$(\ln{q})_{\mathrm{av}}(\theta_{1}\tilde{\omega}) =
(\ln{q})_{\mathrm{av}}(\tilde{\omega})$, we have that the
left\nobreakdash-\hspace{0pt}hand side of the above inequality is,
for all $\omega \in \bar{\Omega}_1^{(1)}$, bounded above by a
constant whose integral over $\Omega$ is not larger than $I$.
Therefore, for any $J \in (I, 0)$ and any $\omega \in
\bar{\Omega}_1^{(1)}$ there is $\tau = \tau(J, \omega) > 0$ such that
\begin{equation*}
\lVert \mathcal{U}_{\theta_{-t}\omega}(t) u -
\mathcal{U}_{\theta_{-t}\omega}(t) \tilde{u} \rVert \le Jt
\end{equation*}
for all $t \ge \tau$ and any $u, \tilde{u} \in \mathcal{S}_1(X^{+})$.
It suffices now to apply the estimate~\eqref{Cauchy-estimate-020} to
$t \in [3, \tau)$ to get the desired result.

A proof of Part (2) is similar: we find an invariant set
$\bar{\Omega}_1^{(2)}$ with $\PP(\bar{\Omega}_1^{(2)}) = 1$ having
the corresponding properties.

The required set $\bar{\Omega}_1$ is defined as $\bar{\Omega}_1^{(1)}
\cap \bar{\Omega}_1^{(2)}$.
\end{proof}

Let $J \in (I, 0)$, where $I$ is as in Proposition
\ref{prop-auxiliary-q}. For any $\omega \in \bar{\Omega}_1$ and $t
\in \TT^{+}$, $3 \le s \le t$ we obtain, via the equality
$\mathcal{U}_{\theta_{-t}\omega}(t) =
\mathcal{U}_{\theta_{-s}\omega}(s) \circ
\mathcal{U}_{\theta_{-t}\omega}(t-s)$, that
\begin{equation*}
\norm{\mathcal{U}_{\theta_{-s}\omega}(s)\mathbf{e} -
\mathcal{U}_{\theta_{-t}\omega}(t)\mathbf{e}} = \rVert
\mathcal{U}_{\theta_{-s}\omega}(s) \mathbf{e} -
\mathcal{U}_{\theta_{-s}\omega}(s)
(\mathcal{U}_{\theta_{-t}\omega}(t-s)\mathbf{e})  \rVert \le
C_1(J,\omega) e^{Js},
\end{equation*}
which allows us to define
\begin{equation}
\label{w-def}
w(\omega) := \lim\limits_{s\to\infty}
\mathcal{U}_{\theta_{-s}\omega}(s)\mathbf{e},
\end{equation}
where the limit is taken in the $X$\nobreakdash-\hspace{0pt}norm.

Since $w(\omega) = \lim_{n\to\infty}
\mathcal{U}_{\theta_{-n}\omega}(n)\mathbf{e}$ and
$\mathcal{U}_{\theta_{-n}\omega}(n)\mathbf{e} \in X^{+}$, we have
that $w(\omega) \in X^{+}$.  Moreover,  as the functions $[\, \omega
\mapsto \mathcal{U}_{\theta_{-n}\omega}(n)\mathbf{e}\, ]$ are
$(\mathfrak{F}, \mathfrak{B}(X))$\nobreakdash-\hspace{0pt}measurable,
$w \colon \bar{\Omega}_1 \to X$ is measurable. We will prove that
$w(\cdot)$ satisfies Theorem \ref{theorem-w}.

\begin{proposition}
\label{w-prop}
\begin{itemize}
\item[{\rm (1)}]
There is a $\tilde{\sigma}_1 > 0$ such that for each $\omega \in
\bar{\Omega}_1$.
\begin{equation*}
\limsup_{\substack{t\to\infty \\ t \in \TT^{+}}} \frac{1}{t}
\ln{\sup{\left\{ \left\lVert \frac{U_{\theta_{-t}\omega}(t)u}
{\norm{U_{\theta_{-t}\omega}(t)u}} - w(\omega) \right\rVert: u \in
X^{+}, \ \norm{u} = 1 \right\}}} \le -\tilde{\sigma}_1.
\end{equation*}
\item[{\rm (2)}]
For each $\omega \in \bar{\Omega}_1$,
\begin{equation*}
\limsup_{\substack{t\to\infty \\ t \in \TT^{+}}} \frac{1}{t}
\ln{\sup{\left\{ \left\lVert
\frac{U_{\omega}(t)u}{\norm{U_{\omega}(t)u}} -
w(\theta_{t}\omega) \right\rVert : u \in X^{+}, \ \norm{u} = 1
\right\} } } \le -\tilde{\sigma}_1.
\end{equation*}
\end{itemize}
\end{proposition}
\begin{proof}
(1) It follows from Proposition \ref{prop-auxiliary-q}(1) and the
definition of $w(\omega)$ that
\begin{equation*}
\norm{\mathcal{U}_{\theta_{-t}\omega}(t)\mathbf{e} - w(\omega)} \le
C_1(J,\omega) e^{Jt} \quad \text{and} \quad \lVert
\mathcal{U}_{\theta_{-t}\omega}(t) \mathbf{e} -
\mathcal{U}_{\theta_{-t}\omega}(t) u \rVert \le C_1(J,\omega) e^{Jt}
\end{equation*}
consequently
\begin{equation*}
\norm{\mathcal{U}_{\theta_{-t}\omega}(t)u - w(\omega)} \le 2
C_1(J,\omega) e^{Jt}
\end{equation*}
for any $\omega \in \bar{\Omega}_1$, any $t \in \TT^{+}$, $t \ge 3$,
and any $u \in \mathcal{S}_1(X^{+})$.  Therefore
\begin{equation*}
\limsup_{\substack{t\to\infty \\ t \in \TT^{+}}} \frac{1}{t}
\ln{\sup{\{\norm{\mathcal{U}_{\theta_{-t}\omega}(t)u - w(\omega)}}: u
\in \mathcal{S}_1(X^{+})\}} \le J,
\end{equation*}
and, since $J > I$ can be taken arbitrarily close to $I$,
\begin{equation*}
\limsup_{\substack{t\to\infty \\ t \in \TT^{+}}} \frac{1}{t}
\ln{\sup{\{\norm{\mathcal{U}_{\theta_{-t}\omega}(t)u(t) -
w(\omega)}}: u \in \mathcal{S}_1(X^{+})\}} \le \int_{\Omega} \ln{q}
\, d\PP
\end{equation*}
for any $\omega \in \bar{\Omega}_1$. (1) thus holds.

\smallskip
(2) is proved just as (1), with Proposition \ref{prop-auxiliary-q}(1)
replaced by Proposition \ref{prop-auxiliary-q}(2).
\end{proof}

\begin{proof} [Proof of Theorem \ref{theorem-w}]

(1) Observe that
\begin{equation*}
\begin{aligned}
\mathcal{U}_{\omega}(t)w(\omega) = {} &
\mathcal{U}_{\omega}(t)(\lim\limits_{s\to\infty}
\mathcal{U}_{\theta_{-s}\omega}(s)\mathbf{e}) \\
= {} & \lim\limits_{s\to\infty}
\mathcal{U}_{\theta_{-s}\omega}(s+t)\mathbf{e} \\
= {} & \lim\limits_{s\to\infty}
\mathcal{U}_{\theta_{-s}(\theta_{t}\omega)}(s)
(\mathcal{U}_{\theta_{-s}\omega}(t)\mathbf{e})
\end{aligned}
\end{equation*}
for any $\omega \in \bar\Omega_1$ and $t \in \TT$.  By Proposition
\ref{prop-auxiliary-q}(1),
\begin{equation*}
\lVert \mathcal{U}_{\theta_{-s}(\theta_{t}\omega)}(s) \mathbf{e} -
\mathcal{U}_{\theta_{-s}(\theta_{t}\omega)}(s)
(\mathcal{U}_{\theta_{-s}\omega}(t)\mathbf{e}) \rVert \le
C_1(J,\theta_{t}\omega) e^{Js},
\end{equation*}
from which it follows that $\mathcal{U}_{\omega}(t)w(\omega) =
w(\theta_{t}\omega)$. (1) is thus proved with $\tilde{\Omega}_1 =
\bar{\Omega}_1$.

For any $\omega \in \bar{\Omega}_1$, as $w(\omega) =
\mathcal{U}_{\theta_{-1}\omega}(1) w(\theta_{-1}\omega)$,
Lemma~\ref{lm-estimate-projective-diameter} implies that $w(\omega)
\in C_{\mathbf{e}}$.

(2) It follows from Proposition \ref{w-prop}(1) that
\begin{equation*}
\lim\limits_{\substack{s \to \infty \\ s\in \TT^{+}}} \biggl\lVert
\mathcal{U}_{\theta_{-s}(\theta_{t}\omega)}(s)
\frac{v_{\omega}(-s+t)} {\norm{v_{\omega}(-s+t)}} -
w(\theta_{t}\omega) \biggr\rVert = 0
\end{equation*}
for each $t \in \TT$ and $\omega \in \bar{\Omega}_1$.

But $\mathcal{U}_{\theta_{-s}\omega}(t) \frac{v_{\omega}(-s+t)}
{\norm{v_{\omega}(-s+t)}}$ is equal, for each $s \in \TT^{+}$, to
$\frac{v_{\omega}(t)}{\norm{v_{\omega}(t)}}$.  Thus we have
$v_{\omega}(t) = \norm{v_{\omega}(t)} w(\theta_{t}\omega)$ for all $t
\in \TT$.  As, for each $t \in \TT$, both $v_{\omega}(t)$ and
$w_{\omega}(t)$ belong to the one\nobreakdash-\hspace{0pt}dimensional
subspace $\tilde{E}_1(\theta_{t}\omega)$, we must have that
$\norm{v_{\omega}(t)}/\norm{w_{\omega}(t)}$ is constant and then
$v_\omega(t) = \norm{v_\omega(0)} w_\omega(t)$.  Hence (2) holds with
$\tilde{\Omega}_1 = \bar{\Omega}_1$.

(3) The mapping
\begin{equation*}
[\, (t,\omega) \ni \TT \times \bar{\Omega}_1 \mapsto
\ln{\rho_{t}(\omega)} \in (-\infty,\infty) \,]
\end{equation*}
is $(\mathfrak{B}(\TT) \otimes \mathfrak{F},
\mathfrak{B}(\RR))$\nobreakdash-\hspace{0pt}measurable. We have
\begin{equation*}
\ln{\rho_{s+t}(\omega)} = \ln{\rho_{t}(\theta_{s}\omega)} +
\ln{\rho_{s}(\omega)}\quad \text{for any}\,\,  s, t \in \TT \
\text{and any}\,\, \omega \in \tilde{\Omega}_1.
\end{equation*}
In the discrete\nobreakdash-\hspace{0pt}time case, the Birkhoff
Ergodic Theorem (Theorem~\ref{Birkhoff-thm}(i)) applied to
$\ln{\rho_1}$ (observe that, by (A1)(i), $\lnplus{\rho_1} \in
L_1(\OFP)$) guarantees the existence of $\tilde{\lambda}_1 \in
[-\infty,\infty)$ and an invariant ${\Omega}_1^{'} \subset
\bar{\Omega}_1$, $\PP(\hat{\Omega}) = 1$ such that
\begin{equation*}
\tilde{\lambda}_1 = \lim\limits_{n\to\pm\infty} \frac{1}{n}
\ln{\rho_{n}(\omega)} = \int\limits_{\Omega} \ln{\rho_{1}}\, d\PP.
\end{equation*}
for all $\omega \in {\Omega}_1^{'}$. (3) then holds with
$\tilde\Omega_1=\Omega_1^{'}$.

In the continuous\nobreakdash-\hspace{0pt}time case, applying the
Birkhoff Ergodic Theorem (Theorem~\ref{Birkhoff-thm}(i)) to $(\OFP,
(\theta_n)_{n \in \ZZ})$ and $\ln{\rho_1}$ we obtain the existence of
$\Omega_1^{''} \subset \bar{\Omega}_1$, $\theta_{1}( \Omega_1^{''}) =
\Omega_1^{''}$, $\PP( \Omega_1^{''}) = 1$, such that
\begin{equation*}
\lim\limits_{n\to\pm\infty} \frac{1}{n} \ln{\rho_{n}(\omega)} =
(\ln{\rho_{1}})_{\mathrm{av}}(\omega)
\end{equation*}
for all $\omega \in \Omega_1^{''}$, where $\int_{\Omega}\ln{\rho_{1}}
\, d\PP = \int_{\Omega}(\ln{\rho_{1}})_{\mathrm{av}} \, d\PP$.  Put
${\Omega}_1^{'''} := \bigcup\limits_{T \in [0,1)}\theta_{T}(
\Omega_1^{''})$.  As $\PP$ is complete, ${\Omega}_1^{'''} \in
\mathfrak{F}$ and $\PP({\Omega}_1^{'''}) = 1$. Using (A1)(i) we prove
that
\begin{equation*}
\lim\limits_{t \to \pm\infty} \frac{1}{t} \ln{\rho_{t}(\omega)} =
\lim\limits_{n \to \pm\infty} \frac{1}{n}
\ln{\rho_{n}(\tilde{\omega})}
\end{equation*}
for any $\omega \in {\Omega}_1^{'''}$ such that $\omega =
\theta_{T}\tilde{\omega}$ for some $\tilde{\omega} \in
\Omega_1^{'''}$ and $T \in [0,1)$ (cf.\ the proof
of~\cite[Lemma~3.4]{Lian-Lu}). Consequently,
$(\ln{\rho_{1}})_{\mathrm{av}}$ is
$\PP$\nobreakdash-\hspace{0pt}a.e.\ constant, so it must be equal to
$\int_{\Omega}\ln{\rho_{1}} \, d\PP$. (3) then holds with
$\tilde\Omega_1=\Omega_1^{'''}$.

\smallskip
(4)  Assume moreover (A1)(ii)--(iii) and that the cone $X^{+}$ is
total.  If Theorem \ref{Oseledets-thm}(1) holds, it is clear that
$\tilde{\lambda}_1 = -\infty$.  If Theorem \ref{Oseledets-thm}(2) or
(3) holds, by Theorem \ref{thm-largest-meets-nontrivially} and (2),
we must have $w(\omega)\in E_1(\omega)$ and hence $\tilde{\lambda}_1
= \lambda_1$. In any case, it follows that for any $u \in X \setminus
\{0\}$,
\begin{equation*}
\limsup_{\substack{t\to\infty \\ t \in \TT^{+}}}
\frac{1}{t}\ln{\norm{U_\omega(t)u}} \le \tilde{\lambda}_1
\end{equation*}
for any $\omega \in \Omega_0 \cap \tilde{\Omega}_1$, where
$\tilde{\Omega}_1$ is as in (3).

\smallskip
(5) Assume moreover (A0)(iii).  By (A0)(iii) and (A2), for any $u \in
X \setminus\{0\}$,
\begin{equation*}
\norm{U_\omega(t) u} \le \norm{U_\omega(t)\abs{u}} \quad \forall
\,\,\omega \in \Omega,\, \, t > 0, \ t \in \TT^{+}.
\end{equation*}
It suffices to prove that
\begin{equation*}
\limsup_{\substack{t\to\infty \\ t \in \TT^{+}}} \frac{1}{t}
\ln{\norm{U_\omega(t)u}} \le \tilde{\lambda}_1
\end{equation*}
for any $u \in X^+ \setminus \{0\}$ and $\omega \in
\tilde{\Omega}_1$, where $\tilde{\Omega}_1$ is as in (3).  By (A3),
for any $u \in X^+ \setminus \{0\}$ and $\omega \in
\tilde{\Omega}_1$,
\begin{equation*}
U_\omega(1) u \le
\frac{\varkappa(\omega)\beta(\omega,u)}{\beta(\omega,w(\omega))}
U_\omega(1) w(\omega)
\end{equation*}
and then, by (A2),
\begin{equation*}
U_\omega(t) u \le
\frac{\varkappa(\omega)\beta(\omega,u)}{\beta(\omega,w(\omega))}
U_\omega(t) w(\omega) \quad \forall t \ge 1, \ t \in \TT^{+}.
\end{equation*}
This together with (A0)(iii) implies that
\begin{equation*}
\norm{U_\omega(t) u} \le
\frac{\varkappa(\omega)\beta(\omega,u)}{\beta(\omega,w(\omega))}
\norm{U_\omega(t) w(\omega)} \quad \forall t \ge 1, \ t \in \TT^{+},
\end{equation*}
and hence
\begin{equation*}
\limsup_{\substack{t\to\infty \\ t \in \TT^{+}}} \frac{1}{t}
\ln{\norm{U_\omega(t)u}} \le \tilde{\lambda}_1.
\end{equation*}
\end{proof}

\begin{proof}[Proof of Theorem \ref{theorem-w-star}]
(1)--(3)  can be proved by arguments similar to those in the proofs
of Theorem \ref{theorem-w}(1)--(3).

\smallskip
(4) For given $\omega \in \tilde{\Omega}_1 \cap \tilde{\Omega}_1^*$
and $t > 0$, $t \in \TT^{+}$, we have
\begin{align*}
\rho^*_t(\omega) \langle
w(\theta_{-t}\omega),w^*(\theta_{-t}\omega)\rangle&
= \langle w(\theta_{-t}\omega), U_\omega^*(t)w^*(\omega) \rangle \\
& =
\langle U_{\theta_{-t}\omega}(t)w(\theta_{-t}\omega),w^*(\omega) \rangle \\
& = \rho_t(\theta_{-t}\omega)\langle w(\omega),w^*(\omega) \rangle \\
& = \frac{1}{\rho_{-t}(\omega)}\langle w(\omega),w^*(\omega) \rangle,
\end{align*}
and
\begin{align*}
\rho_t(\omega) \langle w(\theta_{t}\omega),w^*(\theta_{t}\omega)
\rangle &
= \langle U_{\omega}(t)w(\omega), w^*(\omega) \rangle \\
& = \langle w(\omega), U^{*}_{\theta_{t}\omega}(t)
w^*(\theta_{t}\omega) \rangle \\
& = \rho^{*}_t(\theta_{t}\omega) \langle w(\omega),w^*(\omega) \rangle \\
& = \frac{1}{\rho^{*}_{-t}(\omega)}\langle w(\omega),w^*(\omega)
\rangle,
\end{align*}
which implies that
\begin{align*}
\tilde{\lambda}_1 & = \lim_{\substack{t\to\infty \\ t \in \TT^{+}}}
\biggl(- \frac{1}{t} \ln{\rho_{-t}(\omega)} + \frac{1}{t}
\ln{\langle w(\omega), w^*(\omega) \rangle} \biggr) \\
& = \lim_{\substack{t\to\infty \\ t \in \TT^{+}}} \biggl( \frac{1}{t}
\ln{\rho^{*}_{t}(\omega)} + \frac{1}{t} \ln{\langle
w(\theta_{-t}\omega),w^*(\theta_{-t}\omega)
\rangle} \biggr) \\
& \le \lim_{\substack{t\to\infty \\ t \in \TT^{+}}} \frac{1}{t}
\ln{\rho^{*}_{t}(\omega)} + \limsup_{\substack{t\to\infty \\ t \in
\TT^{+}}} \frac{1}{t} \ln{\langle w(\theta_{-t}\omega),
w^*(\theta_{-t}\omega) \rangle} \\
& = \tilde{\lambda}^{*}_1 + \limsup_{\substack{t\to\infty \\ t \in
\TT^{+}}} \frac{1}{t} \ln{\langle w(\theta_{-t}\omega),
w^*(\theta_{-t}\omega) \rangle},
\end{align*}
and
\begin{align*}
\tilde{\lambda}^{*}_1 & = \lim_{\substack{t\to\infty \\ t \in
\TT^{+}}} \biggl(- \frac{1}{t} \ln{\rho^{*}_{-t}(\omega)} +
\frac{1}{t}
\ln{\langle w(\omega), w^*(\omega) \rangle} \biggr) \\
& = \lim_{\substack{t\to\infty \\ t \in \TT^{+}}} \biggl( \frac{1}{t}
\ln{\rho_{t}(\omega)} + \frac{1}{t} \ln{\langle
w(\theta_{t}\omega),w^*(\theta_{t}\omega)
\rangle} \biggr) \\
& \le \lim_{\substack{t\to\infty \\ t \in \TT^{+}}} \frac{1}{t}
\ln{\rho_{t}(\omega)} + \limsup_{\substack{t\to\infty \\ t \in
\TT^{+}}} \frac{1}{t} \ln{\langle w(\theta_{t}\omega),
w^*(\theta_{t}\omega) \rangle} \\
& = \tilde{\lambda}_1 + \limsup_{\substack{t\to\infty \\ t \in
\TT^{+}}} \frac{1}{t} \ln{\langle w(\theta_{t}\omega),
w^*(\theta_{t}\omega) \rangle}.
\end{align*}
It is enough now to note that $0 \le \langle w(\omega), w^{*}(\omega)
\rangle \le 1$ for any $\omega \in \tilde{\Omega}_1 \cap
\tilde{\Omega}_1^*$.
\end{proof}

\subsection{Generalized exponential separation and proof of Theorem
\ref{separation-thm}}

In this subsection, we study the attractivity properties of
generalized principal Floquet subspaces and prove
Theorem~\ref{separation-thm}.  To this end, we first prove some
auxiliary results.  Throughout this subsection, we assume (A0)(iii),
(A0)$^*$(iii), (A1)$^*$(i), (A2)$^*$, and (A4).

The next result gives the formula for the projection of $X$ on
$\tilde{F}_1(\omega)$ along $\tilde{E}_1(\omega)$.
\begin{lemma}
\label{projection-formula-lemma}
The family $\{\tilde{P}(\omega)\}_{\omega \in \tilde{\Omega}_1 \cap
\tilde{\Omega}^{*}_1}$ of projections associated with the
decomposition $\tilde{E}_1(\omega) \oplus \tilde{F}_1(\omega) = X$ is
given by the formula
\begin{equation}
\label{projection-formula}
\tilde{P}(\omega)u = u - \frac{\langle u, w^{*}(\omega)
\rangle}{\langle w(\omega), w^{*}(\omega) \rangle}w(\omega), \qquad
\omega \in \tilde{\Omega}_1 \cap \tilde{\Omega}^{*}_1.
\end{equation}
\end{lemma}
\begin{proof}
Fix $\omega \in \tilde{\Omega}_1 \cap \tilde{\Omega}^{*}_1$.  After
simple computation it is clear that $u \in \tilde{E}_1(\omega)$ if
and only~if $\tilde{P}(\omega)u = 0$, and that $\tilde{P}(\omega)u =
u$ if and only~if $u \in \tilde{F}(\omega)$.
\end{proof}

\begin{lemma}
\label{lemma-positive-inner-product}
For each $\omega \in \tilde{\Omega}_1 \cap \tilde{\Omega}_1^{*}$
there holds
\begin{equation*}
w(\omega) \ge \frac{\mathbf{e}}{\varkappa(\theta_{-1}\omega)} \qquad
\text{and} \qquad w^{*}(\omega) \ge
\frac{\mathbf{e}^{*}}{\varkappa^{*}(\theta_{1}\omega)}.
\end{equation*}
\end{lemma}
\begin{proof}
By (A3) and (A3)$^*$, for each $\omega \in \tilde{\Omega}_1 \cap
\tilde{\Omega}_1^{*}$ there are $\tilde{\beta}, \tilde{\beta}^{*} >
0$ such that
\begin{equation*}
\tilde{\beta} \mathbf{e} \le w(\omega) \le \varkappa(\theta_{-1}\omega)
\tilde{\beta} \mathbf{e} \qquad \text{and} \qquad \tilde{\beta}^{*}
\mathbf{e}^{*} \le w^{*}(\omega) \le \varkappa^{*}(\theta_{1}\omega)
\tilde{\beta}^{*} \mathbf{e}^{*}.
\end{equation*}
Since $\norm{w(\omega)} = \norm{w^{*}(\omega)} = \norm{\mathbf{e}} =
\norm{\mathbf{e}^{*}} = 1$, we have $\tilde{\beta} \le 1 \le
\tilde{\beta} \varkappa(\theta_{-1}\omega)$ and $\tilde{\beta}^{*}
\le 1 \le \tilde{\beta}^{*} \varkappa^{*}(\theta_{1}\omega)$,
 hence
\begin{equation*}
w(\omega) \ge \frac{\mathbf{e}}{\varkappa(\theta_{-1}\omega)} \qquad
\text{and} \qquad w^{*}(\omega) \ge
\frac{\mathbf{e}^{*}}{\varkappa^{*}(\theta_{1}\omega)}.
\end{equation*}
\end{proof}

For $\omega \in \tilde{\Omega}^{*}_1$ we define
\begin{align*}
W^+(\omega) & := \{\, u \in X: \langle u, w^{*}(\omega) \rangle
> \,0\},\\
W^-(\omega) & := \{\,u \in X: \langle u,w^{*}(\omega) \rangle <
0 \,\}.
\end{align*}
Observe that, by~\eqref{dual-definition},
\begin{equation*}
\langle U_{\omega}(t)u, v^{*} \rangle = \langle u,
U^{*}_{\theta_{t}\omega}(t)v^{*} \rangle
\end{equation*}
for each $t \ge 0$.  Hence $U_\omega(t)W^+(\omega) \subset
W^+(\theta_{t}\omega)$ and $U_\omega(t)W^-(\omega) \subset
W^-(\theta_{t}\omega)$.
\begin{lemma}
\label{lemma-cone-in-W-plus}
For each $\omega \in \tilde{\Omega}_1 \cap \tilde{\Omega}_1^{*}$
there holds $X^{+} \setminus \{0\} \subset W^{+}(\omega)$.
\end{lemma}
\begin{proof}
Let $\omega \in \tilde{\Omega}_1 \cap \tilde{\Omega}_1^{*}$ and let
$u \in X^{+} \setminus \{0\}$.  We have
\begin{equation*}
\begin{aligned}
\langle u, w^{*}(\omega) \rangle  = {} &
\frac{1}{\norm{U^{*}_{\theta_{1}\omega}(1)w^{*}(\theta_{1}\omega)}}
\langle u, U^{*}_{\theta_{1}\omega}(1)w^{*}(\theta_{1}\omega) \rangle
& \text{by Theorem~\ref{theorem-w-star}(1)} \\
= {} & \frac{1}{\rho_1^{*}(\theta_{1}\omega)} \langle u,
U^{*}_{\theta_{1}\omega}(1)w^{*}(\theta_{1}\omega) \rangle & \text{by
the definition of $\rho_1^{*}$} \\
= {} & \frac{1}{\rho_1^{*}(\theta_{1}\omega)} \langle U_{\omega}(1)u,
w^{*}(\theta_{1}\omega) \rangle & \text{by~\eqref{dual-definition}} \\
\ge {} & \frac{\beta(\omega,u)}{\rho_1^{*}(\theta_{1}\omega)}
\langle \mathbf{e}, w^{*}(\theta_{1}\omega) \rangle & \text{by (A3)} \\
\ge {} & \frac{\beta(\omega,u)}
{\varkappa^{*}(\theta_{2}\omega)\,\rho_1^{*}(\theta_{1}\omega)} \langle
\mathbf{e}, \mathbf{e}^{*} \rangle & \text{by
Lemma~\ref{lemma-positive-inner-product}} \\
> {} & 0. & \text{by (A4)}
\end{aligned}
\end{equation*}
\end{proof}

\begin{proposition}
\label{tempered}
The function
\begin{equation*}
[\, \tilde{\Omega}_1 \cap \tilde{\Omega}_1^{*} \ni \omega \mapsto
\ln{\langle w(\omega), w^{*}(\omega) \rangle} \in (-\infty,0] \,]
\end{equation*}
belongs to $L_1(\OFP)$.
\end{proposition}
\begin{proof}
By Lemma~\ref{lemma-positive-inner-product},
\begin{equation*}
0 < \frac{\langle \mathbf{e}, \mathbf{e}^{*}
\rangle}{\varkappa(\theta_{-1}\omega) \varkappa^{*}(\theta_{1}\omega)} \le
\langle w(\omega), w^{*}(\omega) \rangle \le 1,
\end{equation*}
which implies that
\begin{equation*}
\ln{\langle \mathbf{e}, \mathbf{e}^{*} \rangle} -
\ln{\varkappa(\theta_{-1}\omega)} - \ln{\varkappa^{*}(\theta_{1}\omega)}
\le \ln{\langle w(\omega), w^{*}(\omega) \rangle} \le 0
\end{equation*}
for all $\omega \in \tilde{\Omega}_1 \cap \tilde{\Omega}_1^{*}$.  It
suffices to apply the fact that both $\ln{\varkappa}$ and
$\ln{\varkappa^{*}}$ belong to $L_1(\OFP)$.
\end{proof}

\begin{proposition}
\label{prop-oscillation-estimate}
There exists an invariant $\tilde{\Omega}_2 \subset \tilde{\Omega}_1
\cap \tilde{\Omega}_1^{*}$, $\PP(\tilde{\Omega}_2) = 1$, with the
property that for each $J$, $\int_{\Omega} \ln{p} \, d\PP < J < 0$,
and each $\omega \in \tilde{\Omega}_2$ there is $C_3(\omega, J)
> 0$ such that
\begin{equation*}
\osc\Bigl(\tfrac{U_{\omega}(t)u}{\rho_t(\omega)}/
w(\theta_{t}\omega)\Bigr) \le C_3(J, \omega) e^{Jt}
\end{equation*}
for all $u \in \mathcal{S}_1(X^{+})$ and all $t \ge 1$, $t \in
\TT^{+}$.
\end{proposition}
\begin{proof}
By (A3),
\begin{equation*}
\beta(\omega,u) \mathbf{e} \le U_{\omega}(1)u \le \varkappa(\omega)
\beta(\omega,u) \mathbf{e}
\end{equation*}
and
\begin{equation*}
\beta(\omega,w(\omega)) \mathbf{e} \le U_{\omega}(1) w(\omega) \le
\varkappa(\omega) \beta(\omega,w(\omega)) \mathbf{e},
\end{equation*}
consequently
\begin{equation*}
\osc(U_{\omega}(1)u/U_{\omega}(1)w(\omega)) \le \left( \varkappa(\omega)
- \frac{1}{\varkappa(\omega)} \right)
\frac{\beta(\omega,u)}{\beta(\omega,w(\omega))}
\end{equation*}
for each nonzero $u \in X^{+}$.  Observe that $\beta(\omega, u)
\mathbf{e} \le U_{\omega}(1)u$ and $U_{\omega}(1)w(\omega) \le
\varkappa(\omega) \beta(\omega, w(\omega)) \mathbf{e}$, which implies
that $\beta(\omega, u) \le \norm{U_{\omega}(1)u} \le
\norm{U_{\omega}(1)}$ for all $u \in \mathcal{S}_1(X^{+})$, and
$\rho_1(\omega) \le \varkappa(\omega) \beta(\omega, w(\omega))$.
 Therefore
\begin{equation*}
\osc(U_{\omega}(1)u/U_{\omega}(1)w(\omega)) \le (\varkappa^{2}(\omega) -
1) \frac{\norm{U_{\omega}(1)}}{\rho_1(\omega)}
\end{equation*}
for all $u \in \mathcal{S}_1(X^{+})$.  The remainder of the proof
goes, with the help of
Lemma~\ref{lm-projective-diameter-nonincreasing} and the equality
\begin{equation*}
\osc\Bigl(\tfrac{U_{\omega}(t)u}{\rho_t(\omega)}/
w(\theta_{t}\omega)\Bigr) =
\osc(U_{\omega}(t)u/U_{\omega}(t)w(\omega)) \qquad u \in
\mathcal{S}_1(X^{+}), \ t \ge 1, \ t \in \TT^{+},
\end{equation*}
along the lines of the proof of Proposition \ref{prop-auxiliary-q}.
\end{proof}

\begin{proposition}
\label{w-w-star-prop}
There is a $\tilde{\sigma}_2 > 0$ such that for each $\omega \in
\tilde{\Omega}_2$ $(\tilde\Omega_2$ is as in Proposition
\ref{prop-oscillation-estimate}$)$ there holds
\begin{equation*}
\limsup_{\substack{t\to\infty \\ t \in \TT^{+}}} \frac{1}{t}
\ln{\sup{\left\{ \left\lVert
\frac{U_{\omega}(t)u}{\rho_{t}(\omega)} - \frac{\langle u,
w^{*}(\omega) \rangle}{\langle w(\omega), w^{*}(\omega) \rangle}
w(\theta_{t}\omega)  \right\rVert: u \in X, \ \norm{u} = 1
\right\}}} \le -\tilde{\sigma}_2.
\end{equation*}
\end{proposition}

\begin{proof} Denote $\tilde{U}_{\omega}(t)u :=
\tfrac{U_{\omega}(t)u}{\rho_t(\omega)}$.  By Proposition
\ref{prop-oscillation-estimate}, there exists an invariant
$\tilde{\Omega}_2 \subset \tilde{\Omega}_1 \cap
\tilde{\Omega}_1^{*}$, $\PP(\tilde{\Omega}_2) = 1$, such that for any
$J \in (I, 0)$, any $\omega \in \tilde{\Omega}_2$ there is $C_3(J,
\omega) > 0$ such that
\begin{equation*}
\osc(\tilde{U}_{\omega}(t)u/w(\theta_{t}\omega)) \le C_3(J, \omega)
e^{Jt}
\end{equation*}
for all $t \in \TT^{+}$, $t \ge 1$, and all $u \in
\mathcal{S}_1(X^{+})$.  Since
\begin{equation*}
\osc(\tilde{U}_{\omega}(t)u/w(\theta_{t}\omega)) =
M(\tilde{U}_{\omega}(t)u/w(\theta_{t}\omega)) -
m(\tilde{U}_{\omega}(t)u/w(\theta_{t}\omega)),
\end{equation*}
it follows via Lemma~\ref{lm-projective-diameter-nonincreasing} that
$m(\tilde{U}_{\omega}(t)u/w(\theta_{t}\omega))$ converges in a
nondecreasing way, as $t \to \infty$, and
$M(\tilde{U}_{\omega}(t)u/w(\theta_{t}\omega))$ converges in a
nonincreasing way, as $t \to \infty$, to a common limit (denoted by
$\mu(u,\omega)$).  Further, we have
\begin{align*}
\mu(u,\omega) - m(\tilde{U}_{\omega}(t)u/w(\theta_{t}\omega)) & {}
\le C_3(J, \omega) e^{Jt}, \\
M(\tilde{U}_{\omega}(t)u/w(\theta_{t}\omega)) - \mu(u,\omega)  & {}
\le C_3(J, \omega) e^{Jt}
\end{align*}
for all $t \in \TT^{+}$, $t \ge 1$, and all $u \in
\mathcal{S}_1(X^{+})$.

As
\begin{align*}
(m(\tilde{U}_{\omega}(t)u/w(\theta_{t}\omega)) - \mu(u,\omega))
w(\theta_{t}\omega) &
\le \tilde{U}_{\omega}(t)u - \mu(u,\omega) w(\theta_{t}\omega) \\
&\le (M(\tilde{U}_{\omega}(t)u/w(\theta_{t}\omega)) - \mu(u,\omega))
w(\theta_{t}\omega),
\end{align*}
there holds
\begin{equation*}
\left\lVert \frac{U_{\omega}(t)u}{\rho_{t}(\omega)} - \mu(u,\omega)
w(\theta_{t}\omega) \right\rVert \le C_3(J, \omega) e^{Jt}
\end{equation*}
for all $t \in \TT^{+}$, $t \ge 1$, and all $u \in
\mathcal{S}_1(X^{+})$.

Fix $u \in \mathcal{S}_1(X^{+})$ and $\omega \in \tilde{\Omega}_2$.
We apply the functionals $w^{*}(\theta_{n}\omega)$ to the
exponentially decaying sequence (in $X$)
\begin{equation*}
\frac{U_{\omega}(n)u}{\rho_{n}(\omega)} - \mu(u,\omega)
w(\theta_{n}\omega)
\end{equation*}
to obtain an exponentially decaying sequence (in $\RR$)
\begin{equation}
\label{exponential-decay}
\begin{aligned}
& \qquad \left\langle \frac{U_{\omega}(n)u}{\rho_{n}(\omega)} -
\mu(u,\omega) w(\theta_{n}\omega), w^{*}(\theta_{n}\omega)
\right\rangle \\
& = \frac{\langle U_{\omega}(n)u, w^{*}(\theta_{n}\omega)
\rangle}{\rho_{n}(\omega)} - \mu(u,\omega) \langle
w(\theta_{n}\omega), w^{*}(\theta_{n}\omega) \rangle.
\end{aligned}
\end{equation}
Observe that
\begin{equation*}
\langle U_{\omega}(n)u, w^{*}(\theta_{n}\omega) \rangle = \langle u,
U^{*}_{\theta_{n}\omega}(n) w^{*}(\theta_{n}\omega) \rangle =
\rho^{*}_n(\theta_{n}\omega) \langle u, w^{*}(\omega) \rangle
\end{equation*}
and
\begin{equation*}
\begin{aligned}
\rho_n(\omega) \langle w(\theta_{n}\omega), w^{*}(\theta_{n}\omega)
\rangle & = \langle U_{\omega}(n)w(\omega), w^{*}(\theta_{n}\omega)
\rangle \\
= \langle w(\omega),
U^{*}_{\theta_{n}\omega}(n)w^{*}(\theta_{n}\omega) \rangle & =
\rho^{*}_n(\theta_{n}\omega) \langle w(\omega), w^{*}(\omega)
\rangle,
\end{aligned}
\end{equation*}
hence the exponentially decaying sequence
in~\eqref{exponential-decay} equals
\begin{equation*}
\left( \frac{\langle u,w^*(\omega) \rangle}{\langle w(\omega),
w^{*}(\omega) \rangle} - \mu(u,\omega) \right) \langle
w(\theta_{n}\omega), w^{*}(\theta_{n}\omega) \rangle.
\end{equation*}
It follows from Proposition~\ref{tempered} that
\begin{equation*}
\lim\limits_{n\to\infty} \frac{1}{n} \ln{\langle w(\theta_{n}\omega),
w^{*}(\theta_{n}\omega) \rangle} = 0,
\end{equation*}
consequently
\begin{equation*}
\frac{\langle u,w^*(\omega) \rangle}{\langle w(\omega), w^{*}(\omega)
\rangle} = \mu(u,\omega).
\end{equation*}

We have thus proved that
\begin{equation*}
\left\lVert \frac{U_{\omega}(t)u}{\rho_{t}(\omega)} - \frac{\langle
u, w^{*}(\omega) \rangle}{\langle w(\omega), w^{*}(\omega) \rangle}
w(\theta_{t}\omega)  \right\rVert \le C_3(J, \omega) \norm{u} e^{Jt}
\end{equation*}
for $\omega \in \tilde{\Omega}_2$, $t \ge 1$, $t \in \TT^{+}$, and $u
\in X^{+}$.

Next, let $u \in \mathcal{S}_1(X)$ be arbitrary.  As $(X, X^{+})$ is
a Banach lattice, there holds $\norm{u^{+}} \le \norm{u} = 1$ and
$\norm{u^{-}} \le \norm{u} = 1$.  Therefore we have
\begin{equation*}
\left\lVert \frac{U_{\omega}(t)u}{\rho_{t}(\omega)} - \mu(u,\omega)
w(\theta_{t}\omega) \right\rVert \le 2 C_3(J, \omega) e^{Jt}
\end{equation*}
for all $t \in \TT^{+}$, $t \ge 1$, and $\omega \in
\tilde{\Omega}_2$.
\end{proof}

\begin{proof} [Proof of Theorem \ref{separation-thm}]
(1)  The strong measurability follows, through the
formula~\eqref{projection-formula}, by the measurability of $w$ and
$w^{*}$.

For $\omega \in \tilde{\Omega}_1 \cap \tilde{\Omega}_1^{*}$ define
\begin{equation*}
\tilde{\tilde{P}}(\omega)u := \frac{\langle u, w^{*}(\omega)
\rangle}{\langle w(\omega), w^{*}(\omega) \rangle} w(\omega), \qquad
u \in X.
\end{equation*}
$\tilde{\tilde{P}}(\omega) = \Id_{X} - \tilde{P}(\omega)$, that is,
it equals the projection of $X$ onto $\tilde{E}_1(\omega)$ along
$\tilde{F}_1(\omega)$.  There holds
\begin{equation*}
1 \le \norm{\tilde{\tilde{P}}(\omega)} \le \frac{1}{\langle
w(\omega), w^{*}(\omega) \rangle},
\end{equation*}
consequently
\begin{equation*}
0 \le \ln{\norm{\tilde{\tilde{P}}(\omega)}} \le -\ln{\langle
w(\omega), w^{*}(\omega) \rangle}.
\end{equation*}
It follows from Proposition~\ref{tempered} that
$\ln{\norm{\tilde{\tilde{P}}(\cdot)}}$ belongs to $L_1(\OFP)$. Hence,
by Lemma~\ref{lemma-zero},
\begin{equation*}
\lim\limits_{\substack{t \to \pm\infty \\ t\in \TT}}
\frac{\ln{\norm{\tilde{\tilde{P}}(\theta_{t}\omega)}}}{t} = 0
\end{equation*}
$\PP$\nobreakdash-\hspace{0pt}a.s.\ on $\tilde{\Omega}_1 \cap
\tilde{\Omega}_1^{*}$. Therefore, for each $\epsilon > 0$ and
$\PP$\nobreakdash-\hspace{0pt}a.e.\ $\omega \in \tilde{\Omega}_1 \cap
\tilde{\Omega}_1^{*}$ there is $D(\epsilon, \omega) > 0$ such that
\begin{equation*}
\norm{\tilde{\tilde{P}}(\theta_{t}\omega)} \le D(\epsilon, \omega)
e^{{\epsilon} \abs{t}}
\end{equation*}
for all $t \in \TT$.  Since $\tilde{P}(\omega) = \Id_{X} -
\tilde{\tilde{P}}(\omega)$, we estimate
\begin{equation*}
\norm{\tilde{P}(\theta_{t}\omega)} \le 1 + D(\epsilon, \omega)
e^{{\epsilon} \abs{t}} \le (1 + D(\epsilon, \omega)) e^{{\epsilon}
\abs{t}}
\end{equation*}
for all $t \in \TT$.  Consequently, $\limsup\limits_{\substack{t \to
\pm\infty \\ t\in \TT}} (1/t) \ln{\norm{\tilde{P}(\theta_{t}\omega)}}
\le \epsilon$, but as $\epsilon > 0$ is arbitrary, we have
$\limsup\limits_{\substack{t \to \pm\infty \\ t\in \TT}} (1/t)
\ln{\norm{\tilde{P}(\theta_{t}\omega)}} \le 0$. The inequality
$\liminf\limits_{\substack{t \to \pm\infty \\ t\in \TT}} (1/t)
\ln{\norm{\tilde{P}(\theta_{t}\omega)}} \ge 0$ follows by the fact
that $\norm{\tilde{P}(\omega)} \ge 1$.

\smallskip

(2) It is a consequence of Lemma \ref{lemma-cone-in-W-plus}.

\smallskip

(3) Fix $\omega \in \tilde{\Omega}_2$.  It is clear that
$\liminf\limits_{\substack{t \to \infty \\ t \in \TT^{+}}} (1/t)
\ln{\norm{U_{\omega}(t)}} \ge \tilde{\lambda}_1$.

For a nonzero $u \in X$, put $u_1 := \tilde{P}(\omega)u$, $u_2 := u -
\tilde{P}(\omega)u$. We have $u_2 = \frac{\langle u, w^*(\omega)
\rangle}{\langle w(\omega), w^{*}(\omega) \rangle} w(\omega)$.  Take
some $J$, $\int_{\Omega} \ln{p}\, d\PP < J < 0$.  It follows from
Proposition \ref{w-w-star-prop} that
\begin{equation*}
\begin{aligned}
\norm{U_{\omega}(t)u_1} & = \Bigl\lVert U_{\omega}(t)u -
\frac{\langle u, w^*(\omega) \rangle}{\langle w(\omega),
w^{*}(\omega) \rangle}
\rho_t(\omega) w(\theta_{t}\omega) \Bigr\rVert \\
& \le C_4(J, \omega) \norm{u} \rho_t(\omega) e^{Jt},
\end{aligned}
\end{equation*}
for all $t \in \TT^{+}$, $t \ge 1$, where $C_4(J, \omega) > 0$.
Consequently,
\begin{equation*}
\norm{U_{\omega}(t)u} \le \norm{U_{\omega}(t)u_2} +
\norm{U_{\omega}(t)u_1} \le \rho_t(\omega) \norm{u} (\norm{\Id_{X} -
\tilde{P}(\omega)} + C_4(J, \omega) e^{Jt}).
\end{equation*}
This implies that $\limsup\limits_{\substack{t \to \infty \\
t \in \TT^{+}}} (1/t) \ln{\norm{U_{\omega}(t)}} \le
\tilde{\lambda}_1$. Hence $
\lim\limits_{\substack{t \to \infty \\
t \in \TT^{+}}} (1/t) \ln{\norm{U_{\omega}(t)}}= \tilde{\lambda}_1. $

Assume now $u \in X \setminus \tilde{F}_1(\omega)$.  This means that
$\langle u, w^*(\omega) \rangle \ne 0$, from which it follows that
$\norm{U_{\omega}(t)u_2} = \rho_t(\omega) \norm{u_2} > 0$ for all $t
\in \TT^{+}$.  We thus estimate
\begin{equation*}
\norm{U_{\omega}(t)u} \ge \norm{U_{\omega}(t)u_2} -
\norm{U_{\omega}(t)u_1} \ge \rho_t(\omega) (\norm{u_2} - C_4(J,
\omega) \norm{u} e^{Jt})
\end{equation*}
for all $t \in \TT^{+}$, $t \ge 1$, which gives
$\liminf\limits_{\substack{t \to \infty \\ t \in \TT^{+}}} (1/t)
\ln{\norm{U_{\omega}(t)u}} \ge \tilde{\lambda}_1$.

If $u \in X^{+} \setminus \{0\}$ we apply
Lemma~\ref{lemma-cone-in-W-plus} to conclude that $u \in X \setminus
\tilde{F}_1(\omega)$. Then by Proposition \ref{w-prop}(2),
$\lim\limits_{\substack{t \to \infty \\ t \in \TT^{+}}} (1/t)
\ln{\norm{U_{\omega}(t)u}} = \tilde{\lambda}_1$ for any $u\in
X^+\setminus\{0\}$.

\smallskip
(4)  For each $n \in \NN$ put
\begin{equation*}
f_n(\omega) :=
\ln{\frac{\norm{U_{\omega}(n)|_{\tilde{F}_1(\omega)}}}
{\norm{U_{\omega}(n) w(\omega)}}}.
\end{equation*}
The functions $f_n$ are $(\mathfrak{F},
\mathfrak{B}(\RR))$\nobreakdash-\hspace{0pt}measurable, with
$(f_1)^{+} \in L_1(\OFP)$ (by (A1)(i) and
Theorem~\ref{theorem-w}(3)). Moreover,
\begin{equation*}
f_{m+n}(\omega) \le f_{m}(\omega) + f_{n}(\theta_{m}\omega), \qquad
m, n \in \NN, \ \omega \in \tilde{\Omega}_2.
\end{equation*}

In the discrete\nobreakdash-\hspace{0pt}time case an application of
the Kingman Subadditive Ergodic Theorem (Theorem~\ref{Kingman-thm})
to $((\tilde{\Omega}_2, \mathfrak{F}, \PP), (\theta_{n})_{n \in
\ZZ})$ and $(f_n)$ gives the existence of an invariant
$\tilde{\Omega}_0 \subset \tilde{\Omega}_2$, $\PP(\tilde{\Omega}_0) =
1$, and of $\tilde{\sigma} \in (-\infty, \infty]$ with the property
that
\begin{equation*}
\lim_{n \to \infty} \frac{1}{n}
\ln{\frac{\norm{U_{\omega}(n)|_{\tilde{F}_1(\omega)}}}
{\norm{U_{\omega}(n) w(\omega)}}} = -\tilde{\sigma}
\end{equation*}
for any $\omega \in \tilde{\Omega}_0$.

In the continuous\nobreakdash-\hspace{0pt}time case an application of
the Kingman Subadditive Ergodic Theorem (Theorem~\ref{Kingman-thm})
to $((\tilde{\Omega}_2, \mathfrak{F}, \PP), (\theta_{n})_{n \in
\ZZ})$ and $(f_n)$ gives the existence of $\Omega' \subset
\tilde{\Omega}_2$, $\theta_1(\Omega') = \Omega'$, $\PP(\Omega') = 1$,
and of an $(\mathfrak{F},
\mathfrak{B}(\RR))$\nobreakdash-\hspace{0pt}measurable function $g
\colon \Omega' \to \RR$ satisfying $g^{+} \in L_1(\OFP)$ and
$\int_{\Omega} g \, d\PP = \allowbreak \lim_{n\to\infty} \allowbreak
(1/n) \int_{\Omega} f_n \, d\PP$.  It follows from (A1)(i) (see the
proof of~\cite[Lemma~3.4]{Lian-Lu}) combined with
Theorem~\ref{theorem-w}(3) that
\begin{equation*}
\lim_{t \to \infty} \frac{1}{t}
\ln{\frac{\norm{U_{\omega}(t)|_{\tilde{F}_1(\omega)}}}
{\norm{U_{\omega}(t) w(\omega)}}}
\end{equation*}
is constant for all $\omega \in \tilde{\Omega}_0 :=
\bigcup\limits_{T\in[0,1]} \theta_{T}(\Omega')$.  The set
$\tilde{\Omega}_0$ is clearly invariant.  Since $\PP$ is complete,
$\tilde{\Omega}_0 \in \mathfrak{F}$ with $\PP(\tilde{\Omega}_0) = 1$.

As a consequence of Proposition~\ref{w-w-star-prop}, $\tilde{\sigma}
\ge \tilde{\sigma}_2 > 0$.

The equality
\begin{equation*}
\lim_{\substack{t\to\infty \\ t \in \TT^{+}}} \frac{1}{t}
\ln{\norm{U_{\omega}(t)|_{\tilde{F}_1(\omega)}}} =
\tilde{\lambda}_2 = \tilde{\lambda}_1 - \tilde{\sigma}, \qquad
\omega \in \tilde{\Omega}_0,
\end{equation*}
is straightforward.

(5) First, if Theorem \ref{Oseledets-thm}(2) or (3) occurs, then by
Theorem \ref{theorem-w}(4), $\lambda_1 = \tilde{\lambda}_1 >
-\infty$.

Observe that, by Parts (3) and (4), $\tilde{F}_1(\omega) \setminus
\{0\}$ is, for each $\omega \in \tilde{\Omega}_0$, characterized as
the set of those nonzero $u \in X$ for which $\limsup_{t \to \infty,
t \in \TT^{+}}(1/t) \ln{\norm{U_{\omega}(t)u}} < \tilde{\lambda}_1$.
By Theorem~\ref{Oseledets-thm}, $\hat{F}_1(\omega) \setminus \{0\}$
is, for each $\omega \in \Omega_0$, characterized as the set of those
nonzero $u \in X$ for which $\limsup_{t \to \infty, t \in
\TT^{+}}(1/t) \ln{\norm{U_{\omega}(t)u}} < \lambda_1$.  Consequently,
$\tilde{F}_1(\omega) = \hat{F}_1(\omega)$ for all $\omega \in
\Omega_0 \cap \tilde{\Omega}_0$.  Further, from the above
characterizations it follows that $\hat{\lambda}_2 =
\tilde{\lambda}_2$.

As $\codim{\hat{F}_1(\omega)} = \codim{\tilde{F}_1(\omega)} = 1$, we
have $\dim{E_1(\omega)} = \dim{\tilde{E}_1(\omega)} = 1$, for any
$\omega \in \Omega_0 \cap \tilde{\Omega}_0$.

By Theorem~\ref{thm-largest-meets-nontrivially}(2), for
$\PP$\nobreakdash-\hspace{0pt}a.e.\ $\omega \in \Omega_0$ there
exists an entire trajectory $v_{\omega}$ of $U_{\omega}$ such that
$v_{\omega}(t) \in (E_1(\theta_{t}\omega) \cap X^{+}) \setminus
\{0\}$ for all $t \in \TT$. Therefore $v_{\omega}(t) \in
\tilde{E}_1(\theta_{t}\omega) \setminus \{0\}$ for all $t \in \TT$.
But $E_1(\theta_{t}\omega) = \spanned\{v_{\omega}(t)\}$, hence
$E_1(\theta_{t}\omega) = \tilde{E}_1(\theta_{t}\omega)$.

(6) Observe that, by Theorem~\ref{theorem-w}(4), $\tilde{\lambda}_1 =
\tilde{\lambda}_1^*$.  Without loss of generality, we assume that
(A5) holds and prove $\tilde{\lambda}_1
> -\infty$.

For each $\omega \in \Omega$ and each $n = 1,2,3, \dots$ we have
\begin{equation*}
\norm{U_{\omega}(n)\mathbf{\bar e}} \ge \nu(\omega) \dots
\nu(\theta_{n-1}\omega).
\end{equation*}
 In the
discrete\nobreakdash-\hspace{0pt}time case the Birkhoff Ergodic
Theorem (Theorem~\ref{Birkhoff-thm}(i)) applied to $-\ln{\nu}$ gives
that for $\PP$\nobreakdash-\hspace{0pt}a.e.\ $\omega \in \Omega$
there holds
\begin{equation*}
\lim\limits_{n\to\infty}\frac{1}{n} \sum\limits_{i=0}^{n-1}
\ln{\nu(\theta_{i}\omega)} = \int\limits_{\Omega} \ln{\nu}\, d\PP >
-\infty.
\end{equation*}

In the continuous\nobreakdash-\hspace{0pt}time case the Birkhoff
Ergodic Theorem (Theorem~\ref{Birkhoff-thm}(i)) applied to $(\OFP,
(\theta_{n})_{n \in \ZZ})$ and to $-\ln{\nu}$, together with (A1)(i)
(again cf.\ the proof of ~\cite[Lemma~3.4]{Lian-Lu}), give that for
$\PP$\nobreakdash-\hspace{0pt}a.e.\ $\omega \in \Omega$ one has
\begin{equation*}
\lim\limits_{t\to\infty}\frac{1}{t}
\ln{\norm{U_{\omega}(t)\mathbf{\bar e}}} \ge
\lim\limits_{n\to\infty}\frac{1}{n} \sum\limits_{i=0}^{n-1}
\ln{\nu(\theta_{i}\omega)} =: (\ln{\nu})_{\mathrm{av}}(\omega)
\end{equation*}
As the left\nobreakdash-\hspace{0pt}hand side is
$\PP$\nobreakdash-\hspace{0pt}a.e.\ constant, it must be $\ge
\int_{\Omega} (\ln{\nu})_{\mathrm{av}} \, d\PP = \int_{\Omega}
\ln{\nu} \, d\PP > -\infty$. Then by Theorem \ref{theorem-w}(4), we
must have $\tilde\lambda_1>-\infty$.
\end{proof}

\subsection{Monotonicity}

In this subsection, we prove Theorem \ref{theorem-comparison}, which
shows that the monotonicity of two measurable skew-product semiflows
at some time implies the monotonicity of the associated generalized
principal Lyapunov exponents.  We assume

\begin{proof}[Proof of Theorem \ref{theorem-comparison}]
Let $\tilde{\Omega}_0^{(i)}$, $i = 1, 2$, be a set of full measure
such that
\begin{equation*}
\lim\limits_{\substack{t \to \infty \\ t \in \TT^{+}}}
\frac{\ln{\norm{U_{\omega}^{(i)}(t)u}}}{t} =
\tilde{\lambda}_1^{(i)}
\end{equation*}
holds for any $\omega \in \tilde{\Omega}_0^{(i)}$ and any $u \in
X^{+} \setminus \{0\}$ (see Theorem~\ref{separation-thm}(3)).  Pick
$\omega \in \tilde{\Omega}_0^{(1)} \cap \tilde{\Omega}_0^{(2)}$ such
that $U_{\omega}^{(1)}(t_0)u \le U_{\omega}^{(2)}(t_0)u$ for all $u
\in X^{+}$.  We have, by the monotonicity of the norm,
\begin{equation*}
\tilde{\lambda}_1^{(1)} = \lim\limits_{n \to \infty}
\frac{\ln{\norm{U_{\omega}^{(1)}(nt_0)u}}}{n} \le \lim\limits_{n \to
\infty} \frac{\ln{\norm{U_{\omega}^{(2)}(nt_0)u}}}{n} =
\tilde{\lambda}_1^{(2)}.
\end{equation*}
\end{proof}

\section*{Acknowledgments}
The authors are grateful to the referee for helpful suggestions,
in~particular for calling their attention to the concept of
$u_0$\nobreakdash-\hspace{0pt}positivity.

\bibliographystyle{amsplain}

\end{document}